\documentclass[11pt,a4paper]{article}
\usepackage{amsfonts}
\usepackage{amssymb}
\usepackage{mathrsfs}
\usepackage{amsmath}
\usepackage{amsthm}
\usepackage{enumerate}

\allowdisplaybreaks
\newtheorem{df}{Definition}[section]
\newtheorem{thm}[df]{Theorem}
\newtheorem{lem}[df]{Lemma}
\newtheorem{prop}[df]{Proposition}
\newtheorem{cor}[df]{Corollary}
\newtheorem{exa}[df]{Example}
\newtheorem{rem}[df]{Remark}
\newcommand{\A}{{\mathcal{A}}}
\begin{document}
\title{{\bf BiHom-Lie colour algebras structures}}
\author{\normalsize \bf  Kadri Abdaoui, Abdelkader Ben Hassine and Abdenacer Makhlouf }

\maketitle

{\bf\begin{center}{Abstract}\end{center}}

 BiHom-Lie Colour algebra is a generalized Hom-Lie Colour algebra endowed with two commuting
multiplicative linear maps. The main purpose of this paper is to define representations and a cohomology of BiHom-Lie colour algebras and to study some key constructions and properties.
 Moreover, we discuss  $\alpha^{k}\beta^l$-generalized derivations,  $\alpha^{k}\beta^l$-quasi-derivations and $\alpha^{k}\beta^l$-quasi-centroid. We provide some properties  and their relationships with  BiHom-Jordan colour algebra.


\renewcommand{\thefootnote}{\fnsymbol{footnote}}
\footnote[0]{ Corresponding author (Abdenacer Makhlouf): Abdenacer.Makhlouf@uha.fr,  Laboratoire de Math\'ematiques, Informatique
et Applications, University of Haute Alsace, 4 rue Fr\`eres Lumi\`ere, 68093 Mulhouse, France.}
\footnote[0]{Abdaoui Kadri: Abdaoui.elkadri@hotmail.com, University of Sfax, Faculty of Sciences,  BP
1171, 3000 Sfax.}
\footnote[0]{Abdelkader Ben Hassine: benhassine.abdelkader@yahoo.fr, University of Sfax, Faculty of Sciences,  BP
1171, 3000 Sfax, Tunisia.}
\section*{Introduction}
   As generalizations of Lie algebras, Hom-Lie algebras were introduced motivated by
applications in Physics and to deformations of Lie algebras, especially Lie algebras of
vector fields.\\
   Hom-Lie colour algebras are the natural generalizations of Hom-Lie algebras and Hom-Lie superalgebras. In recent years, they have become an interesting subject of mathematics and physics. A cohomology  of  Lie colour algebras were introduced and investigated in \cite{scheunert1979theory,scheunert1979generalized}  and representations of  Lie colour algebras were explicitly described in \cite{feldvoss2001represcolor}.
 Hom-Lie colour algebras were studied first  in \cite{yuan2010hom}, while   in the particular case of Hom-Lie superalgebras, a cohomology theory   was provided in \cite{sadaoui2011cohomology}. Notice that  for Hom-Lie algebras, cohomology was described in \cite{AEM,makhlouf2007notes,sheng2010representations} and representations also in \cite{benayadi2010hom}.


A BiHom-algebra is an algebra in such a way that the identities defining the structure are twisted by two homomorphisms
$\alpha$,$\beta$. This class of algebras was introduced from a categorical approach in \cite{Gra-Mak-Men-Pan} as an extension of the class of Hom-algebras.
More applications of  BiHom-Lie algebras, BiHom-algebras, BiHom-Lie superalgebras and BiHom-Lie admissible superalgebras can be found in \cite{Chen-Qi,Wang-Guo}.

In the present article, we introduce and study the BiHom-Lie colour algebras,
which can be viewed as an extension of BiHom-Lie (super)algebras to $\Gamma$-graded
algebras, where $\Gamma$ is any abelian group.

The paper is organized as follows. In Section 1, we recall  definitions and some key constructions of  BiHom-Lie colour algebras and provide a list of twists of  BiHom-Lie colour algebras. In Section $2$ we introduce a multiplier $\sigma$ on the abelian group $\Gamma$ and provide
constructions of new BiHom-Lie colour algebras using the twisting action of the
multiplier $\sigma$. We show that the $\sigma$-twist of any BiHom-Lie colour algebra is still a BiHom-Lie colour algebra.

%

In Section $3$, we extend the classical concept of Lie admissible algebras to BiHom-
Lie colour settings. Hence, we obtain a more generalized algebra class called BiHom-Lie
colour admissible algebras. We also explore some other general class of algebras: $G$–
BiHom-associative colour algebras, where $G$ is any subgroup of the symmetric group
$S_3$, using which we classify all the BiHom-Lie colour admissible algebras.\\
In Section $4$, we construct a family of cohomologies of BiHom-Lie colour algebras, discuss representation theory in connection with cohomology. 
In the last section, we discuss homogeneous $\alpha^{k}\beta^l$-generalized derivations and the $\alpha^{k}\beta^l$-centroid of BiHom-Lie colour algebras. We generalize to Hom-setting the results  obtained  in \cite{LIn ni}. Moreover, We show that $\alpha^{-1}\beta^2$-derivations of BiHom-Lie colour algebras give rise to BiHom-Jordan colour algebras.
\section{Definitions, proprieties and Examples}
In the following we summarize definitions of BiHom-Lie and BiHom-associative colour algebraic
structures generalizing the well known Hom-Lie and Hom-associative colour algebras.\\
Throughout the article we let $\mathbb{K}$ be an algebraically closed field of characteristic $0$ and
$\mathbb{K^{\ast}}=\mathbb{K}\backslash \{0\}$ be the group of units of $\mathbb{K}$.

Let $\Gamma$ be an abelian group. A vector space $V$ is said to be $\Gamma$-graded, if there is a family
$(V_{\gamma})_{\gamma\in \Gamma}$ of vector subspace of $V$ such that $$V=\bigoplus_{\gamma\in \Gamma}V_{\gamma}.$$
An element $x \in V$ is said to be homogeneous of  degree $\gamma \in \Gamma$ if $x \in V_{\gamma}, \gamma\in \Gamma$, and in this case, $\gamma$ is called the degree of $x$. As usual, we denote by $\overline{x}$ the degree of an element $x \in V$. Thus each homogeneous
element $x \in V$ determines a unique group of element  $\overline{x} \in \Gamma$ by $x \in V_{\overline{x}}$. Fortunately, We can drop
the symbol $"-"$, since confusion rarely occurs.  In the sequel, we will denote by $\mathcal{H}(V)$ the set of all the homogeneous elements of $V$.\\
Let $V=\bigoplus_{\gamma\in \Gamma}V_{\gamma}$ and $V^{'}=\bigoplus_{\gamma\in \Gamma}V^{'}_{\gamma}$ be two $\Gamma$-graded vector spaces. A linear mapping $f: V \longrightarrow V^{'}$  is said to be homogeneous of  degree $\upsilon \in \Gamma$ if
  $f(V_{\gamma})\subseteq V^{'}_{\gamma+\upsilon}, ~~ \forall \gamma \in \Gamma.$
If in addition $f$ is  homogeneous of degree zero, i.e. $f(V_{\gamma})\subseteq V^{'}_{\gamma}$ holds for any $\gamma \in \Gamma$, then $f$ is said to be even.

An algebra $\mathcal{A}$ is said to be $\Gamma$-graded if its underlying vector space is $\Gamma$-graded, i.e. $\mathcal{A}=\bigoplus_{\gamma\in \Gamma}\mathcal{A}_{\gamma}$, and if, furthermore $\mathcal{A}_{\gamma}\mathcal{A}_{\gamma'}\subseteq \mathcal{A}_{\gamma+\gamma'}$, for all $\gamma, \gamma'\in \Gamma$. It is easy to see
that if $\mathcal{A}$ has a unit element $e$, it follows that $e \in \mathcal{A}_{0}$. A subalgebra of $\mathcal{A}$ is said to be graded if it is graded as a subspace of $\mathcal{A}$.\\
Let $\mathcal{A}^{'}$ be another $\Gamma$-graded algebra. A homomorphism $f:\mathcal{A} \longrightarrow \mathcal{A}^{'}$ of $\Gamma$-graded algebras is by definition a homomorphism  of the algebra $\mathcal{A}$ into the algebra  $\mathcal{A}^{'}$, which is, in addition an even mapping.\\
\begin{df} Let $\mathbb{K}$ be a field and  $\Gamma$ be an abelian group. A map $\varepsilon:\Gamma\times\Gamma\rightarrow \mathbb{K^{\ast}}$ is called a skewsymmetric \textit{bicharacter} on ${\Gamma}$ if the following identities hold, for all $a,b,c$ in $\Gamma$
\begin{enumerate}
\item~~$\varepsilon(a,b)~\varepsilon(b,a)=1,$
\item~~$\varepsilon(a,b+c)=\varepsilon(a,b)~\varepsilon(a,c),$
\item~~$\varepsilon(a+b,c)=\varepsilon(a,c)~\varepsilon(b,c).$
\end{enumerate}
\end{df}
 The definition above implies, in particular, the following relations
$$\varepsilon(a,0)=\varepsilon(0,a)=1,\ \varepsilon(a,a)=\pm1, \  \textrm{for\ all}\  a \in \Gamma.$$
If $x$ and $x'$ are two homogeneous elements of degree $\gamma$ and $\gamma'$ respectively and $\varepsilon$ is a skewsymmetric bicharacter, then we shorten the notation by writing $\varepsilon(x,x')$ instead of $\varepsilon(\gamma,\gamma')$.
\begin{df} A \textit{BiHom-Lie colour} algebra is a $5$-uple $(\mathcal{A},[.,.],\varepsilon,\alpha,\beta)$ consisting of a $\Gamma$-graded vector space $\mathcal{A}$, an even bilinear mapping
$[.,.]:\mathcal{A}\times\mathcal{A}\rightarrow\mathcal{A}$ $($i.e. $[\mathcal{A}_{a},\mathcal{A}_{b}]\subseteq \mathcal{A}_{a+b}$ for all $a,b \in \Gamma)$, a bicharacter $\varepsilon:\Gamma\times\Gamma\rightarrow \mathbb{K^{\ast}}$ and two even homomorphism $\alpha,\beta:\mathcal{A}\rightarrow\mathcal{A}$ such that for homogeneous elements $x,y,z$ we have
\begin{itemize}
  \item[] $\alpha \circ \beta = \beta \circ \alpha,$
  \item[] $\alpha ([x,y])=[\alpha(x),\alpha(y)],~~\beta ([x,y])=[\beta(x),\beta(y)],$
  \item[] $[\beta(x),\alpha(y)]=-\varepsilon(x,y)[\beta(y),\alpha(x)],$
  \item[] $\circlearrowleft_{x,y,z}\varepsilon(z,x)[\beta^2(x),[\beta(y),\alpha(z)]]= 0~~(\varepsilon\textrm{-BiHom-Jacobi\ condition})$
\end{itemize}
where $\circlearrowleft_{x,y,z}$ denotes summation over the cyclic permutation on $x,y,z$.
\end{df}

Obviously, a Hom-Lie colour algebra $(\mathcal{A},[.,.],\varepsilon,\alpha)$ is a particular case of BiHom-Lie colour algebra. Conversely a BiHom-Lie colour algebra $(\mathcal{A},[.,.],\varepsilon,\alpha,\alpha)$ with bijective $\alpha$ is a Hom-Lie colour algebra $(\mathcal{A},[.,.],\varepsilon,\alpha)$.
\begin{rem} A  Lie colour algebra $(\mathcal{A},[.,.],\varepsilon)$ is a Hom-Lie colour algebra with $\alpha=Id$, since the $\varepsilon$-Hom-Jacobi condition
reduces to the $\varepsilon$-Jacobi condition when $\alpha=Id$. If, in addition, $\varepsilon(x,y)=1$ or $\varepsilon(x,y)=(-1)^{|x||y|}$, then the BiHom-Lie colour algebra is  a classical Hom-Lie algebra or Hom-Lie superalgebra. Using  definitions of \cite{Gra-Mak-Men-Pan, Wang-Guo},  BiHom-Lie algebras and BiHom-Lie superalgebras
are also obtained when $\varepsilon(x,y)=1$ and $\varepsilon(x,y)=(-1)^{|x||y|}$ respectively.
\end{rem}

\begin{df}
\begin{enumerate}\item[]
\item A BiHom-Lie colour algebra $(\mathcal{A},[.,.],\varepsilon,\alpha,\beta)$ is multiplicative if $\alpha$ and $\beta$ are  even algebra morphisms, i.e.,
for any homogenous elements $x, y\in \mathcal{A}$, we have
$$
\alpha([x,y])=[\alpha(x),\alpha(y)]~~~~~~\text{and}~~~~\beta([x,y])=[\beta(x),\beta(y)].
$$
\item A BiHom-Lie colour algebra $(\mathcal{A},[.,.],\varepsilon,\alpha,\beta)$ is regular if $\alpha$ and $\beta$ are  even algebra automorphisms.
\end{enumerate}
\end{df}

We recall in the following the definition of BiHom-associative colour algebra which provide a different way for constructing BiHom-Lie colour algebra by extending the fundamental construction of  Lie colour algebras from associative colour algebra via commutator bracket multiplication.
\begin{df}\label{assoc} A \textit{BiHom-associative colour} algebra is a $5$-tuple $(\mathcal{A},\mu,\varepsilon,\alpha,\beta)$ consisting of a $\Gamma$-graded vector space $\mathcal{A}$, an even bilinear map $\mu:\mathcal{A}\times \mathcal{A}\rightarrow \mathcal{A}$ $($i.e. $\mu(\mathcal{A}_a,\mathcal{A}_b)\subset \mathcal{A}_{a+b})$ and two even homomorphisms $\alpha,\beta:\mathcal{A}\rightarrow \mathcal{A}$ such that
$\alpha\circ \beta=\beta \circ \alpha$
\begin{equation}\label{HACA}
    \alpha(x)(yz)=(xy)\beta(z).
\end{equation}
In the case where $xy=\varepsilon(x,y)yx$, the BiHom-associative colour algebra $(\mathcal{A},\mu,\alpha,\beta)$ is called commutative BiHom-associative colour algebra.
\end{df}
Obviously, a Hom-associative colour algebra $(\mathcal{A},\mu,\alpha)$ is a particular case of a BiHom-associative colour algebra,
namely $(\mathcal{A},\mu,\alpha,\alpha)$. Conversely, a BiHom-associative colour algebra $(\mathcal{A},\mu,\alpha,\alpha)$ with bijective $\alpha$ is
the Hom-associative colour algebra $(\mathcal{A},\mu,\alpha)$.
\begin{prop} Let $(\mathcal{A},\mu,\varepsilon,\alpha,\beta)$ be a BiHom-associative colour algebra defined on the vector space $\mathcal{A}$ by the multiplication $\mu$ and  two bijective homomorphisms $\alpha$ and $\beta$. Then the quadruple $(\mathcal{A},[.,.],\varepsilon,\alpha,\beta)$, where the bracket is defined for $x, y \in \mathcal{H}(\mathcal{A})$ by
$$[x,y]=xy-\varepsilon(x,y)(\alpha^{-1}\beta(y))(\alpha\beta^{-1}(x)),$$
is a BiHom-Lie colour algebra.
\end{prop}
\begin{proof}First we check that the bracket product $[\cdot,\cdot]$ is compatible with the structure
maps $\alpha$ and $\beta$. For any homogeneous elements $x, y \in \mathcal{A}$, we have
\begin{eqnarray*}
[\alpha(x),\alpha(y)]&=& \alpha(x)\alpha(y)-\varepsilon(x,y)(\alpha^{-1}\beta(\alpha(y)))(\alpha\beta^{-1}(\alpha(x)))\\
&=& \alpha(x)\alpha(y)-\varepsilon(x,y)\beta(y)(\alpha^2\beta^{-1}(x))\\
&=& \alpha([x,y]).
\end{eqnarray*}
The second equality holds since $\alpha$ is even and $\alpha\circ \beta=\beta\circ \alpha$. Similarly, one can prove that
$\beta([x, y])= [\beta(x),\beta(y)]$.\\
It is easy to verify that $[\beta(x),\alpha(y)]=-\varepsilon(x,y)[\beta(y),\alpha(x)]$.\\
Now we prove the $\varepsilon$-BiHom-Jacobi condition. For any homogeneous elements $x, y,z \in \mathcal{A}$, we have
\begin{eqnarray*}
\varepsilon(z,x)[\beta^2(x),[\beta(y),\alpha(z)]]&=& \varepsilon(z,x)[\beta^2(x),\beta(y)\alpha(z)-\varepsilon(y,z)\alpha^{-1}\beta(\alpha(z))\alpha\beta^{-1}(\beta(y))]\\
&=&\varepsilon(z,x)[\beta^2(x),\beta(y)\alpha(z)]-\varepsilon(z,x)\varepsilon(y,z)[\beta^2(x),\alpha^{-1}\beta(\alpha(z))\alpha\beta^{-1}(\beta(y))]\\
&=& \varepsilon(z,x)\Big(\beta^2(x)(\beta(y)\alpha(z))-\varepsilon(x,y+z)(\alpha^{-1}(\beta^2(y))\beta(z))\alpha(\beta(x)) \Big)\\
&-& \varepsilon(z,x)\varepsilon(y,z)\Big( \beta^2(x)(\beta(z)\alpha(y))-\varepsilon(x,z+y)(\alpha^{-1}(\beta^2(z))\beta(y))\alpha(\beta(x)) \Big).
\end{eqnarray*}
Similarly, we have
\begin{eqnarray*}
\varepsilon(x,y)[\beta^2(y),[\beta(z),\alpha(x)]]
&=&\varepsilon(x,y)\Big(\beta^2(y)(\beta(z)\alpha(x))-\varepsilon(y,z+x)(\alpha^{-1}(\beta^2(z))\beta(x))\alpha(\beta(y)) \Big)\\
&-& \varepsilon(x,y)\varepsilon(z,x)\Big( \beta^2(y)(\beta(x)\alpha(z))-\varepsilon(y,z+x)(\alpha^{-1}(\beta^2(x))\beta(z))\alpha(\beta(y)) \Big).
\end{eqnarray*}
\begin{eqnarray*}
\varepsilon(y,z)[\beta^2(z),[\beta(x),\alpha(y)]]
&=& \varepsilon(y,z)\Big(\beta^2(z)(\beta(x)\alpha(y))-\varepsilon(z,x+y)(\alpha^{-1}(\beta^2(x))\beta(y))\alpha(\beta(z)) \Big)\\
&-& \varepsilon(y,z)\varepsilon(x,y)\Big( \beta^2(z)(\beta(y)\alpha(x))- \varepsilon(z,x+y)(\alpha^{-1}(\beta^2(y))\beta(x))\alpha(\beta(z)) \Big).
\end{eqnarray*}
By the associativity, we have
\begin{eqnarray*}
\circlearrowleft_{x,y,z}\varepsilon(z,x)[\beta^2(x),[\beta(y),\alpha(z)]]&=& 0.
\end{eqnarray*}
And this finishes the proof.
\end{proof}

The following theorem generalizes the result of \cite{yuan2010hom}. In the following,  starting from a BiHom-Lie colour algebra and an even Lie colour algebra endomorphism, we construct a new BiHom-Lie colour algebra. We say that it is obtained by twisting principle or composition method.

\begin{thm}\label{induced-color} Let $(\mathcal{A},[.,.],\varepsilon)$ be an ordinary Lie colour algebra and let $\alpha,\beta:\mathcal{A}\longrightarrow\mathcal{A}$ two commuting even linear maps such that $\alpha([x,y])=[\alpha(x),\alpha(y)]$ and $\beta([x,y])=[\beta(x),\beta(y)]$, for all $ x,y\in \mathcal{H}(\mathcal{A})$. Define the even linear map $\{\cdot,\cdot\}:\mathcal{A}\times \mathcal{A}\longrightarrow \mathcal{A}$
$$
  \{x,y\}=[\alpha(x),\beta(y)],~~\forall~~x,y\in \mathcal{H}(\mathcal{A}).
$$
Then $\mathcal{A}_{(\alpha,\beta)}=(\mathcal{A},\{\cdot,\cdot\},\varepsilon,\alpha,\beta)$ is a BiHom-Lie colour algebra.
\end{thm}
\begin{proof} Obviously $\{.,.\}$ is BiHom-$\varepsilon$-skewsymmetric. Furthermore $(\mathcal{A},\{.,.\},\varepsilon,\alpha,\beta)$ satisfies the $\varepsilon$-BiHom-Jacobi condition. Indeed  
\begin{eqnarray*}
   \circlearrowleft_{x,y,z}\varepsilon(z,x)\{\beta^2(x),\{\beta(y),\alpha(z)\}\} &=&
   \circlearrowleft_{x,y,z}\varepsilon(z,x)\{\beta^2(x),[\alpha\beta(y),\beta\alpha(z)]\} \\
   &=& \circlearrowleft_{x,y,z}\varepsilon(z,x)[\alpha\beta^2(x),[\alpha\beta^2(y),\alpha\beta^2(z)]] \\
   &=&  0.
\end{eqnarray*}
\end{proof}

\textbf{Claim:} More generally, let $(\mathcal{A},[.,.],\varepsilon,\alpha,\beta)$ be a BiHom-Lie colour algebra and
$\alpha^{'},\beta^{'}:\mathcal{A}\longrightarrow \mathcal{A}$ even linear maps such that $\alpha^{'}([x,y])=[\alpha^{'}(x),\alpha^{'}(y)]$ and $\beta^{'}([x,y])=[\beta^{'}(x),\beta^{'}(y)],$ for all $x,y\in \mathcal{H}(\mathcal{A})$, and any two of the maps $\alpha,\beta,\alpha^{'},\beta^{'}$ commute. Then $(\mathcal{A},[\cdot,\cdot]_{(\alpha^{'},\beta^{'})}=[\cdot,\cdot]\circ (\alpha^{'} \otimes \beta^{'}),\varepsilon,\alpha \circ \alpha^{'},\beta\circ \beta^{'})$ is a BiHom-Lie colour algebra.
\begin{rem}Let $(\mathcal{A},[.,.],\varepsilon)$ be a Lie colour algebra and $\alpha$ be a  Lie colour algebra morphism, then $(\mathcal{A},[.,.]_{\alpha}=\alpha \circ [.,.],\varepsilon,\alpha)$ is a multiplicative  Hom-Lie colour algebra.
\end{rem}

\begin{exa}We construct an example of a BiHom-Lie colour algebra, which is not a Lie colour algebra
starting from the orthosymplectic Lie superalgebra. We consider in the sequel the matrix realization
of this Lie superalgebra.\\
 Let $osp(1, 2) = \mathcal{A}_0 \oplus \mathcal{A}_1$ be the Lie superalgebra where $\mathcal{A}_0$ is spanned by
\begin{eqnarray*}
&& H=\left(
  \begin{array}{ccc}
    1 & 0 & 0 \\
    0 & 0 & 0 \\
    0 & 0 & -1 \\
  \end{array}
\right),~~X=\left(
  \begin{array}{ccc}
    0 & 0 & 1 \\
    0 & 0 & 0 \\
    0 & 0 & 0 \\
  \end{array}
\right),~~Y=\left(
  \begin{array}{ccc}
    0 & 0 & 0 \\
    0 & 0 & 0 \\
    1 & 0 & 0 \\
  \end{array}
\right).
\end{eqnarray*}
and $\mathcal{A}_1$ is spanned by
\begin{eqnarray*}
&& F=\left(
  \begin{array}{ccc}
    0 & 0 & 0 \\
    1 & 0 & 0 \\
    0 & 1 & 0 \\
  \end{array}
\right),~~G=\left(
  \begin{array}{ccc}
    0 & 1 & 0 \\
    0 & 0 & -1 \\
    0 & 0 & 0 \\
  \end{array}
\right).
\end{eqnarray*}
The defining relations $($we give only the ones with non-zero values in the right-hand side$)$ are
\begin{eqnarray*}
&& [H,X]= 2X,\hskip1cm[H,Y] = -2Y,\hskip1cm [X,Y] = H,\\
&& [Y,G]= F,\hskip1cm[X,F] = G,\hskip1cm [H,F]=-F,\hskip1cm [H,G] = G,\\
&& [G,F]= H,\hskip1cm [G,G]=-2X,\hskip1cm [F,F]= 2Y.
\end{eqnarray*}
Let $\lambda,\kappa \in \mathbb{R}^\ast$, we consider the linear maps $\alpha_\lambda: osp(1,2)\longrightarrow osp(1,2)$ and $\beta_\kappa: osp(1,2)\longrightarrow osp(1,2)$ defined by
\begin{eqnarray*}
\alpha_\lambda(X)&=& \lambda^2 X,~~\alpha_\lambda(Y)=\frac{1}{\lambda^2}Y,~~\alpha_\lambda(H)=H,~~\alpha_\lambda(F)=\frac{1}{\lambda}F,~~ \alpha_\lambda(G)=\lambda G,\\
\beta_\kappa(X)&=& \kappa^2 X,~~\beta_\kappa(Y)=\frac{1}{\kappa^2}Y,~~\beta_\kappa(H)=H,~~\beta_\kappa(F)=\frac{1}{\kappa}F,~~ \beta_\kappa(G)=\kappa G.
\end{eqnarray*}
Obviously, we have $\alpha_\lambda \circ \beta_\kappa=\beta_\kappa \circ \alpha_\lambda$. For all $H,X,Y,F$ and $G$ in $osp(1,2)$, we have
\begin{eqnarray*}
\alpha_\lambda([H,X])&=& \alpha_\lambda(2X)=2 \lambda^2 X,~~
\alpha_\lambda([H,Y])= \alpha_\lambda(-2Y)=-2\frac{1}{\lambda^2}Y,\\
\alpha_\lambda([X,Y])&=& \alpha_\lambda(H)=H,~~
\alpha_\lambda([Y,G])= \alpha_\lambda(F)=\frac{1}{\lambda}F,\\
\alpha_\lambda([X,F])&=& \alpha_\lambda(G)=\lambda G,~~
\alpha_\lambda([H,F])= \alpha_\lambda(-F)=-\frac{1}{\lambda}F,\\
\alpha_\lambda([H,G])&=& \alpha_\lambda(G)=\lambda G,~~
\alpha_\lambda([G,F])= \alpha_\lambda(H)=H,\\
\alpha_\lambda([G,G])&=& \alpha_\lambda(-2X)=-2\lambda^2 X,~~
\alpha_\lambda([F,F])= \alpha_\lambda(2Y)=2\frac{1}{\lambda^2}Y.
\end{eqnarray*}
On the other hand, we have
\begin{eqnarray*}
&&[\alpha_\lambda(H),\alpha_\lambda(X)]=[H,\lambda^2 X]=2\lambda^2 X,~~
[\alpha_\lambda(H),\alpha_\lambda(Y)]= [H,\frac{1}{\lambda^2}Y]=-2\frac{1}{\lambda^2}Y,\\
&&[\alpha_\lambda(X),\alpha_\lambda(Y)]=[\lambda^2 X,\frac{1}{\lambda^2}Y]=H,~~
[\alpha_\lambda(Y),\alpha_\lambda(G)]=[\frac{1}{\lambda^2}Y,\lambda G]=\frac{1}{\lambda}F,\\
&&[\alpha_\lambda(X),\alpha_\lambda(F)]=[\lambda^2 X,\frac{1}{\lambda}F]=\lambda G,~~
[\alpha_\lambda(H),\alpha_\lambda(F)]=[H,\frac{1}{\lambda}F]=-\frac{1}{\lambda}F,\\
&&[\alpha_\lambda(H),\alpha_\lambda(G)]=[H,\lambda G]=\lambda G,~~
[\alpha_\lambda(G),\alpha_\lambda(F)]= [\lambda G,\frac{1}{\lambda}F]=H,\\
&&[\alpha_\lambda(G),\alpha_\lambda(G)]=[\lambda G,\lambda G]=-2\lambda^2 X,~~
[\alpha_\lambda(F),\alpha_\lambda(F)]= [\frac{1}{\lambda}F,\frac{1}{\lambda}F]=2\frac{1}{\lambda^2}Y.
\end{eqnarray*}
Therefore, for $a,a^{'}\in osp(1,2)$, we have
\begin{eqnarray*}
\alpha_\lambda([a,a^{'}])&=& [\alpha_\lambda(a),\alpha_\lambda(a^{'})].
\end{eqnarray*}
Similarly, we have
\begin{eqnarray*}
\beta_\kappa([a,a^{'}])&=& [\beta_\kappa(a),\beta_\kappa(a^{'})].
\end{eqnarray*}
Applying Theorem \eqref{induced-color} we obtain a family of BiHom-Lie colour algebras $osp(1,2)_{\alpha_\lambda,\beta_\kappa} = \Big(osp(1,2),\{\cdot,\cdot\}=[\cdot,\cdot]\circ (\alpha_\lambda \otimes \beta_\kappa),\alpha_\lambda,\beta_\kappa \Big)$ where the BiHom-Lie colour algebra bracket $\{\cdot,\cdot\}$
 on the basis elements is given, for $\lambda,\kappa \neq 0$, by
\begin{eqnarray*}
&& \{H,X\}= 2\kappa^{2}X,\hskip1cm \{H,Y\} = \frac{-2}{\kappa^{2}}Y,\hskip1cm \{X,Y\} =(\frac{\lambda}{\kappa})^2 H,\\
&& \{Y,G\}=\frac{\kappa}{\lambda^2} F,\hskip1cm \{X,F\} =\frac{\lambda^2}{\kappa} G,\hskip1cm \{H,F\}=-\frac{1}{\kappa}F,\hskip1cm \{H,G\} =\kappa G,\\
&& \{G,F\}=\frac{\lambda}{\kappa} H,\hskip1cm \{G,G\}=-2\lambda \kappa X,\hskip1cm\{F,F\}= 2\frac{\lambda}{\kappa}Y.
\end{eqnarray*}
These BiHom-Lie colour algebras are not Hom-Lie colour algebras for $\lambda\neq 1$.
Indeed, the left-hand side of the $\varepsilon$-Hom-Jacobi identity, for $\beta_\kappa= id$, leads to
\begin{eqnarray*}
\{\alpha_\lambda(X),\{Y,H\}\}-\{\alpha_\lambda(H),\{X,Y\}\}+\{\alpha_\lambda(Y),\{H,X\}\}&=& 2\frac{\lambda^{6}-1}{\lambda^4}H,
\end{eqnarray*}
and also
\begin{eqnarray*}
\{\alpha_\lambda(H),\{F,F\}\}-\{\alpha_\lambda(F),\{H,F\}\}+\{\alpha_\lambda(F),\{F,H\}\}&=& 4\frac{1-\lambda}{\lambda^2}Y,
\end{eqnarray*}
Then, they do not vanish for $\lambda \neq 1$.
\end{exa}


\begin{exa}
Let $(A,[~,~]_A,\varepsilon, \alpha,\beta)$ be a BiHom-Lie colour algebra. Then the vector space $A':=A\otimes\mathbb{K}[t,t^{-1}]$ can be considered as the algebra of Laurent polynomials
with coefficients in the BiHom-Lie colour algebra $A$. Note that $A'$ can be endowed with
a natural $\Gamma$-grading as follows: an element $x\in A'$ is said to be homogeneous of
degree $a\in  \Gamma$, if there exist an element $x_a\in A$ with degree $a$ and $f(t)\in\mathbb{K}[t,t^{-1}]$, such that $x=x_a\otimes f(t)$. Put $\alpha'=\alpha\otimes Id_A$,  $\beta'=\beta\otimes Id_A$ and define an even bilinear multiplication $[~,~]_{A'}$ on $A'$ by
$$
[x\otimes f(t),y\otimes g(t)]=[x,y]\otimes f(t)g(t)
$$
for all $x, y\in H(A)$ and $f(t), g(t)\in \mathbb{K}[t,t^{-1}]$. Then $(A',[~,~]_{A'},\varepsilon,\alpha',\beta')$ is a BiHom-Lie colour algebra. Indeed, For any homogeneous elements
$x, y, z\in A$, and $f(t), g(t), h(t)\in \mathbb{K}[t,t^{-1}]$, we have
\begin{eqnarray*}
[\beta'(x\otimes f(t)),\alpha'(y\otimes g(t))]
 &=& [\beta(x)\otimes f(t),\alpha(y)\otimes g(t)] \\
   &=&  [\beta(x),\alpha(y)]\otimes f(t)g(t)\\
   &=&-\varepsilon(x,y)[\beta(y),\alpha(x)]\otimes f(t)g(t)=-[\beta(x),\alpha(x)]\otimes f(t)g(t)\\
   &=&[\beta'(y\otimes g(t)),\alpha'(x\otimes f(t))]
\end{eqnarray*}
and
\begin{eqnarray*}
   \circlearrowleft_{x,y,z}\varepsilon(z,x)[\beta'^2(x\otimes f(t)),[\beta'(y\otimes g(t)),\alpha'(z\otimes h(t))]] &&  \\
 = \circlearrowleft_{x,y,z}\varepsilon(z,x)[\beta^2(x)\otimes f(t),[\beta(y)\otimes g(t),\alpha(z)\otimes h(t)]] &&\\
 =\circlearrowleft_{x,y,z}\varepsilon(z,x)[\beta^2(x),[\beta(y),\alpha(z)]]\otimes f(t)g(t)h(t)=0,
\end{eqnarray*}
since  $(A,[~,~]_A,\varepsilon, \alpha,\beta)$ is a BiHom-Lie colour algebra.
\end{exa}

\section{$\sigma$-Twists of BiHom-Lie colour  algebras}
In this section, we shall give a close relationship   between BiHom-Lie colour algebras corresponding to different form $\sigma$ on $\Gamma$.

Let $(\mathcal A, [~,~], \varepsilon,\alpha,\beta)$ be a BiHom-Lie colour algebra. Given any mapping $\sigma:\Gamma\times \Gamma\rightarrow \mathbb{K^{\ast}}$, we define on the $\Gamma$-graded vector space $\mathcal A$ a new multiplication $[~,~]^\sigma$ by the requirement that
\begin{equation}\label{sigma twist1}
    [x,y]^\sigma=\sigma(x,y)[x,y],
\end{equation}
for all the homogenous elements $x$,$y$ in $\mathcal A$.  The $\Gamma$-graded vector space $\mathcal A$, endowed with the multiplication $[~,~]^\sigma$, is a $\Gamma$-graded algebra that will be called a $\sigma-$twist and will be denoted by $\mathcal A^\sigma$. We will looking for conditions on $\sigma$, which  ensure that $(\mathcal A^\sigma, [~,~]^\sigma,\varepsilon, \alpha,\beta)$ is also a  BiHom-Lie colour algebra.

The bilinear mapping $[~,~]^\sigma$  is a $\varepsilon$-skewsymmetric if and only if
\begin{enumerate}
\item $\sigma$ is symmetric, i.e. $\sigma(x,y)=\sigma(y,x)$, for any $x,y\in\Gamma$.
\item Furthermore, the product $[~,~]^\sigma$ satisfies the $\varepsilon$-BiHom-Jacobi\ condition if and only if
$\sigma(x,y)\sigma(z,x+y)$  is invariant under cyclic permutations of $x,y,z\in\Gamma$.
\end{enumerate}
We call such a mapping $\sigma:\Gamma\times \Gamma\rightarrow \mathbb{K^{\ast}}$ satisfying both $(1)$ et $(2)$  a symmetric multiplier on $\Gamma$. Then we have:
\begin{prop}
With the above notations. Let $(\mathcal A,[~,~],\varepsilon, \alpha,\beta)$ be a BiHom-Lie colour algebra and suppose that $\sigma$ is a symmetric multiplier on $\Gamma$. Then the $\sigma$-twist $(\mathcal A^\sigma, [~,~]^\sigma,\varepsilon, \alpha,\beta)$ is also a BiHom-Lie colour algebra under the same twisting $\varepsilon$.
\end{prop}

\begin{proof}
 For any homogeneous elements $x, y \in \mathcal{A}$, we have
\begin{eqnarray*}
[\alpha(x),\alpha(y)]^\sigma&=& \sigma(\alpha(x),\alpha(y))[\alpha(x),\alpha(y)]\\
&=& \sigma(x,y)\alpha([x,y])\\
&=& \alpha([x,y]^\sigma).
\end{eqnarray*}
Similarly, one can prove that
$\beta([x, y]^\sigma)= [\beta(x),\beta(y)]^\sigma$.\\
Since $\sigma$ is symmetric, we have
$$
[\beta(x),\alpha(y)]^\sigma=\sigma(x,y)[\beta(x),\alpha(y)]=-\sigma(y,x)\varepsilon(x,y)[\beta(y),\alpha(x)]=-\varepsilon(x,y)[\beta(y),\alpha(x)]^\sigma.
$$
Now we prove the $\varepsilon$-BiHom-Jacobi condition. For any homogeneous elements $x, y \in \mathcal{A}$, we have
\begin{eqnarray*}
\varepsilon(z,x)[\beta^2(x),[\beta(y),\alpha(z)]^\sigma]^\sigma&=& \varepsilon(z,x)[\beta^2(x),\sigma(y,z)[\beta(y),\alpha(z)]]^\sigma\\
&=&\varepsilon(z,x)\sigma(y,z)\sigma(x,y+z)[\beta^2(x),[\beta(y),\alpha(z)]].
\end{eqnarray*}
Using (2), we have $\circlearrowleft_{x,y,z}\varepsilon(z,x)[\beta^2(x),[\beta(y),\alpha(z)]^\sigma]^\sigma=0$.
\end{proof}

\begin{cor}
Let $(\mathcal A',[~,~]',\varepsilon',\alpha,\beta)$ be a second BiHom-Lie colour algebra and $\sigma$ be a symmetric multiplier on $\Gamma$. Let $f:\mathcal A\rightarrow \mathcal A'$ be a homomorphism of BiHom-Lie colour algebra, so $f$ is also a homomorphism of BiHom-Lie colour algebras $(\mathcal A^\sigma, [~,~]^\sigma,\varepsilon, \alpha,\beta)$ into $(\mathcal A'^\sigma, [~,~]'^\sigma,\varepsilon', \alpha,\beta)$.
\end{cor}
\begin{proof}
We have $f\circ \alpha=\alpha\circ f$ and $f\circ \beta=\beta\circ f$. For any homogenous element $x,y\in\mathcal A$, we have
\begin{eqnarray*}
  f([x,y]^\sigma) = f(\sigma(x,y)[x,y])
   = \sigma(x,y) f([x,y]) =\sigma(x,y)[f(x),f(y)]'
   =[f(x),f(y)]'^\sigma.
\end{eqnarray*}

\end{proof}

\begin{rem}
It is easy to construct a large class of symmetric multipliers on $\Gamma$ as follows. Let $\omega$ be an arbitrary mapping of $\Gamma$. Then the map $\tau:\Gamma\times\Gamma\rightarrow\mathbb{K^{\ast}}$ defined by
$$
\tau(x,y)=\omega(x+y)\omega(x)^{-1}\omega(y)^{-1}, \forall x,y\in \Gamma,
$$
is a symmetric multiplier on $\Gamma$.

Let  $\sigma:\Gamma\times\Gamma\rightarrow\mathbb{K^{\ast}}$ be a map endowing $\mathcal A$ with a new
multiplication defined by \eqref{sigma twist1}, we define a mapping $\delta:\Gamma\times\Gamma\rightarrow\mathbb{K^{\ast}}$ by
\begin{equation}\label{sigma twist2}
    \delta(x,y)=\sigma(x,y)\sigma(y,x)^{-1},~~\forall x,y\in \Gamma.
\end{equation}
Then it follows that
\begin{equation}\label{sigma twist3}
    [x,y]^\sigma=-\varepsilon(x,y)\delta(x,y)[y,x]^\sigma
\end{equation}
for any homogenous $x, y\in \mathcal A$.

In \cite{scheunert1979generalized}, M. Scheunert provided the necessary and sufficient condition on $\sigma$  which ensure that $\varepsilon\delta$ is a bicharacter on $\Gamma$ where $\varepsilon\delta(x,y)=\varepsilon(x,y)\delta(x,y)$ for all $x,y\in\Gamma$. It turns out that  $\mathcal A^\sigma$ is a $\Gamma$-graded $\varepsilon\delta$-Lie algebra if and only if
\begin{equation}\label{sigma twist4}
    \sigma(x,y+z)\sigma(y,z)=\sigma(x,y)\sigma(x+y,z),~~\forall x,y,z\in \Gamma.
\end{equation}
\end{rem}

For any multiplier $\sigma$, the mapping $\delta$ defined by Equation \eqref{sigma twist2} is a bicharacter on $\Gamma$, which is said to be associated with $\sigma$. Note that $\delta(x,x)=1$ and it follows that from equation \eqref{sigma twist4} that $\sigma(0,x)=\sigma(x,0)\sigma(0,0)$ for all $x\in \Gamma$.

\begin{prop}

 Let $(\mathcal A, [~,~],\varepsilon,\alpha,\beta)$ be a BiHom-Lie colour algebra. Let $\delta$ be the bicharacter on $\Gamma$ associated with a given multiplier $\sigma$. Then $(\mathcal A^\sigma, [~,~]^\sigma,\varepsilon\delta,\alpha,\beta)$ is a BiHom-Lie colour algebra.
\end{prop}
\begin{proof}
Using  Equation \eqref{sigma twist4} and the fact that $\sigma$ is symmetric, it follows that
$$
[\beta(x),\alpha(y)]^\sigma=-\varepsilon(x,y)\sigma(x,y)[\beta(y),\alpha(x)]^\sigma.
$$
Now for any homogenous elements $x,y,z\in\mathcal A$, one has
$$
\circlearrowleft_{x,y,z}\varepsilon\delta(z,x)[\beta^2(x),[\beta(y),\alpha(z)]^\sigma]^\sigma=\sigma(x,y)\sigma(y,z)\sigma(z,x)\circlearrowleft_{x,y,z}\varepsilon(z,x)[\beta^2(x),[\beta(y),\alpha(z)]]=0.
$$
Then  $(\mathcal A^\sigma, [~,~]^\sigma,\varepsilon\delta,\alpha,\beta)$ is a BiHom-Lie colour algebra.
\end{proof}
\begin{cor}

 Let $(\mathcal A', [~,~]',\varepsilon,\alpha,\beta)$ be a second BiHom-Lie colour algebra. Suppose we are given a multiplier $\sigma$ on $\Gamma$; let $\delta$ be the bicharacter on $\Gamma$  associated with it. If $f:\mathcal A\rightarrow\mathcal A'$ is a homomorphism of BiHom-Lie colour algebras, then $f$ is also a homomorphism of BiHom-Lie colour algebra $(\mathcal A^\sigma, [~,~]^\sigma,\varepsilon\delta,\alpha,\beta)$ into $(\mathcal A'^\sigma, [~,~]'^\sigma,\varepsilon\delta,\alpha,\beta)$.
\end{cor}

\section{BiHom-Lie colour admissible algebras}
In this section, we aim to extend the notions and results about Lie admissible (colour)
algebras \cite{Goze2004admissible,yuan2010hom} to more generalized cases: BiHom-Lie colour admissible algebras and flexible
BiHom-Lie colour admissible algebras. We will also explore some other general classes of such kind
of algebras: $G$-BiHom-associative colour algebras, which we use to classify   BiHom-Lie colour
admissible algebras. In this section, we always assume that the structure maps $\alpha$ and $\beta$ are bijective.

\begin{df} Let $(\mathcal{A},\mu,\varepsilon,\alpha,\beta)$ be a BiHom-Lie colour algebra on the $\Gamma$-graded vector space $\mathcal{A}$ defined
by an even mutiplication $\mu$ and an algebra endomorphism $\alpha$ and $\beta$. Let $\varepsilon$ be a bicharacter on $\Gamma$.
Then $(\mathcal{A},[.,.],\varepsilon,\alpha,\beta)$ is said to be a BiHom-Lie colour admissible algebra if the bracket defined by
$$[x, y] = xy -\varepsilon(x, y)(\alpha^{-1}\beta(y)) (\alpha\beta^{-1}(x))$$
satisfies the BiHom-Jacobi identity for all homogeneous elements $x, y, z \in \mathcal{A}$.
\end{df}

Let $(\mathcal{A},[.,.],\varepsilon,\alpha,\beta)$ be a BiHom-Lie colour algebra. Define a new commutator product $[.,.]^{'}$ by
$$ [x,y]^{'}=[x,y]-\varepsilon(x,y)[\alpha^{-1}\beta(y),\alpha\beta^{-1}(x)],~~\forall~~ x, y \in \mathcal{H}(\mathcal{A}).$$
It is easy to see that $[\beta(x),\alpha(y)]^{'}= -\varepsilon(x,y)[\beta(y),\alpha(x)]^{'}$. Moreover, we have
\begin{eqnarray*}
  && \circlearrowleft_{x,y,z}\varepsilon(z,x)[\beta^2(x),[\beta(y),\alpha(z)]^{'}]^{'} \\
  &=& \circlearrowleft_{x,y,z} \varepsilon(z,x)
  [\beta^2(x),[\beta(y),\alpha(z)]-\varepsilon(y,z)[\alpha^{-1}(\beta^2(z)),\beta^{-1}(\alpha^2(y))]]^{'}\\
  &=& \circlearrowleft_{x,y,z} \varepsilon(z,x)\Big( [\beta^2(x),[\beta(y),\alpha(z)]]^{'}-\varepsilon(y,z)[\beta^2(x),[\alpha^{-1}(\beta^2(z)),\beta^{-1}(\alpha^2(y))]]^{'}\Big)\\
  &=& \circlearrowleft_{x,y,z} \varepsilon(z,x)\Big( [\beta^2(x),[y,z]]-\varepsilon(x,y+z)[\alpha^{-1}\beta[\beta(y),\alpha(z)],\beta^{-1}\alpha(\beta^2(x))]-\varepsilon(y,z)[\alpha(x),[z,y]]\\
  &+& \varepsilon(y,z)\varepsilon(x,z+y)[[z,y],\alpha(x)]\Big)\\
  &=& 4 \circlearrowleft_{x,y,z} \varepsilon(z,x)[\beta^2(x),[\beta(y),\alpha(z)]]=0.
\end{eqnarray*}
Our discussion above now shows:
\begin{prop}Any BiHom-Lie colour algebra is BiHom-Lie colour admissible.
\end{prop}
Let $(\mathcal{A},\mu,\varepsilon,\beta,\alpha)$ be a colour BiHom-algebra, that is respectively a vector space $\mathcal{A}$, a multiplication $\mu$,    a bicharacter $\varepsilon$ on the abelian group $\Gamma$ and two linear maps $\alpha$ and $\beta$. Notice  that there is no conditions  required on the given data. Let
$$[x,y]=xy-\varepsilon(x,y)(\alpha^{-1}\beta(y)) (\beta^{-1}\alpha(x))$$
be the associated colour commutative. A BiHom associator $as_{\alpha,\beta}$ of $\mu$ is defined by
\begin{equation}\label{BiHomAssociator}
 as_{\alpha,\beta}(x,y,z)=\alpha(x)(yz)-(xy)\beta(z),~~\forall~~x,y,z\in \mathcal{H}(\mathcal{A}).
 \end{equation}
A colour BiHom-algebra is said to be flexible if $as_{\alpha,\beta}(x,y,x)=0,$ for all $x,y\in \mathcal{H}(\mathcal{A}).$\\
Now let us introduce the notation:
\begin{eqnarray*}
S(x,y,z)&=& \circlearrowleft_{x,y,z} \varepsilon(z,x)as_{\alpha,\beta}(\alpha^{-1}\beta^2(x),\beta(y),\alpha(z)).
\end{eqnarray*}
Then we have the following properties:

\begin{lem}\label{admiss}
\begin{eqnarray*}
S(x,y,z)&=&\circlearrowleft_{x,y,z} \varepsilon(z,x)[\beta^2(x),\beta(y)\alpha(z)].
\end{eqnarray*}
\end{lem}
\begin{proof}The assertion follows by  expanding the commutators on the right hand side:
\begin{eqnarray*}
   && \varepsilon(z,x)[\beta^2(x),\beta(y)\alpha(z)]+\varepsilon(x,y)[\beta^2(y),\beta(z)\alpha(x)]+\varepsilon(y,z)[\beta^2(z),\beta(x)\alpha(y)]\\
   && = \varepsilon(z,x)\beta^2(x)(\beta(y)\alpha(z))-\varepsilon(x,y)(\alpha^{-1}\beta^2(y)\beta(z))\alpha(\beta(x))
   +\varepsilon(x,y)\beta^2(y)(\beta(z)\alpha(x))\\
   && -\varepsilon(y,z)(\alpha^{-1}\beta^2(z)\beta(x))\alpha(\beta(y))
   +\varepsilon(y,z)\beta^2(z)(\beta(x)\alpha(y))-\varepsilon(z,x)(\alpha^{-1}\beta^2(x)\beta(y))\alpha(\beta(z))\\
   && = \circlearrowleft_{x,y,z}\varepsilon(z,x)as_{\alpha,\beta}(\alpha^{-1}\beta^2(x),\beta(y),\alpha(z))\\
   && = S(x,y,z).
\end{eqnarray*}
\end{proof}
\begin{prop}A colour BiHom-algebra $(\mathcal{A},\mu,\varepsilon,\beta,\alpha)$ is BiHom-Lie colour admissible if and only if
it satisfies
\begin{eqnarray*}
S(x, y, z) &=& \varepsilon(x, y)\varepsilon(y, z)\varepsilon(z, x)S(x, z, y),~~\forall x, y, z \in \mathcal{H}(\mathcal{A}).
\end{eqnarray*}
\end{prop}
\begin{proof}From Lemma \ref{admiss}, for all $x,y,z \in \mathcal{H}(\mathcal{A})$ we have
\begin{eqnarray*}
   && S(x,y,z)-\varepsilon(x,y) \varepsilon(y,z) \varepsilon(z,x)S(x,z,y)\\
   &&= \circlearrowleft_{x,y,z}\varepsilon(z,x)[\beta^2(x),\beta(y)\alpha(z)]-\varepsilon(x,y) \varepsilon(y,z) \varepsilon(y,x)\circlearrowleft_{x,y,z}\varepsilon(y,x)[\beta^2(x),\beta(z)\alpha(y)]\\
   && = \circlearrowleft_{x,y,z}\varepsilon(z,x)[\beta^2(x),\beta(y)\alpha(z)-\varepsilon(y,z)\beta(z)\alpha(y)]\\
   && =\circlearrowleft_{x,y,z}\varepsilon(z,x)[\beta^2(x),[\beta(y),\alpha(z)]]=0,
\end{eqnarray*}
which proves the result.
\end{proof}
In the following, we explore some other general classes of BiHom-Lie colour admissible algebras, $G$-Hom-associative colour algebras, extending the class of Hom-associative algebras
and Hom-associative superalgebras. We will provide a classification of BiHom-Lie colour admissible algebras using the symmetric group $S_3$, whereas it was classified in \cite{makhlouf2008silvi} and \cite{ammar2010hom} for the
Hom-Lie and Hom-Lie super cases, respectively.
Let $S_3$ be the symmetric group generated by $\sigma_1 = (1~~2), \sigma_2 = (2~~3)$. Let $(\mathcal{A},\mu,\beta,\alpha)$
be a colour BiHom-algebra with respect to bicharcater $\varepsilon$. Suppose that $S_3$ acts on $\mathcal{A}^3$ in the usual way, i.e. $\sigma(x_1, x_2, x_3) =
(x_{\sigma(1)}, x_{\sigma(2)}, x_{\sigma(3)})$, for any $\sigma\in S_3$ and $x_1, x_2, x_3 \in \mathcal{H}(\mathcal{A})$. For convenience, we introduce a
notion of  degree of the transposition $\sigma_i$ with $i \in \{1, 2\}$ by setting
\begin{eqnarray*}
\varepsilon (\sigma; (x_1, x_2, x_3))=|\sigma_i(x_1, x_2, x_3)| &=& \varepsilon(x_i, x_{i+1}),~~\forall x_1, x_2, x_3 \in \mathcal{H}(\mathcal{A}).
\end{eqnarray*}
It is natural to assume that the degree of the identity is $1$ and for the composition $\sigma_i\circ \sigma_j$ , it
is defined by:
\begin{eqnarray*}
|\sigma_i\circ \sigma_j(x_1, x_2, x_3)| &=& |\sigma_j(x_1, x_2, x_3)||\sigma_i(\sigma_j(x_1, x_2, x_3))|\\
&=& |\sigma_j(x_1, x_2, x_3)||\sigma_i(x_{\sigma_j(1)}, x_{\sigma_j(2)}, x_{\sigma_j(3)})|.
\end{eqnarray*}
One can define by induction the degree for any composition. Hence we have
\begin{eqnarray*}
|Id(x_1, x_2, x_3)| &=& 1,\\
|\sigma_1(x_1, x_2, x_3)|&=&\varepsilon(x_1, x_2),\\
|\sigma_2(x_1, x_2, x_3)|&=&\varepsilon(x_2, x_3),\\
|\sigma_1\circ \sigma_2(x_1, x_2, x_3)|&=&\varepsilon(x_2, x_3)\varepsilon(x_1, x_3),\\
|\sigma_2\circ \sigma_1(x_1, x_2, x_3)|&=&\varepsilon(x_1, x_2)\varepsilon(x_1, x_3),\\
|\sigma_2\circ \sigma_1\circ \sigma_2(x_1, x_2, x_3)|&=&\varepsilon(x_2, x_3)\varepsilon(x_1, x_3)\varepsilon(x_1, x_2),\\
\end{eqnarray*}
for any homogeneous elements $x_1, x_2, x_3$ in $\mathcal{A}$. Let $sgn(\sigma)$ denote the signature of $\sigma\in S_3$.
We have the following useful lemma:
\begin{lem}A colour BiHom-algebra $(\mathcal{A},\mu,\varepsilon,\beta,\alpha)$ is a BiHom-Lie admissible colour
algebra if and only if the following condition holds:
\begin{eqnarray*}
\sum_{\sigma \in S_3}(-1)^{sgn(\sigma)}(-1)^{|\sigma(x_1, x_2, x_3)|} as_{\alpha,\beta}\circ \sigma(\alpha^{-1}\beta^2(x_1),\beta(x_2),\alpha(x_3))&=& 0.
\end{eqnarray*}
\end{lem}
\begin{proof}One only needs to verify the BiHom-$\varepsilon$-Jacobi identity. By straightforward calculation,
the associated color commutator satisfies
\begin{eqnarray*}
  && \circlearrowleft_{x_1,x_2,x_3}\varepsilon(x_3,x_1)[\beta^2(x_1),[\beta(x_2),\alpha(x_3)]]\\
  &&=\sum_{\sigma \in S_3}(-1)^{sgn(\sigma)}(-1)^{|\sigma(x_1, x_2, x_3)|}as_{\alpha,\beta}\circ \sigma(\alpha^{-1}\beta^2(x_1),\beta(x_2),\alpha(x_3)).\\
\end{eqnarray*}
\end{proof}
Let $G$ be a subgroup of $S_3$, any  colour BiHom-algebra $(\mathcal{A},\mu,\varepsilon,\beta,\alpha)$ is said to be $G$-BiHom-associative if the following equation holds:
\begin{eqnarray*}
\sum_{\sigma \in G}(-1)^{sgn(\sigma)}(-1)^{|\sigma(x_1, x_2, x_3)|} as_{\alpha,\beta}\circ \sigma(\alpha^{-1}\beta^2(x_1),\beta(x_2),\alpha(x_3))&=& 0,~~\forall~~x_1,x_2,x_3 \in \mathcal{H}(\mathcal{A}).
\end{eqnarray*}
\begin{prop}Let $G$ be a subgroup of the symmetric group $S_3$. Then any $G$-BiHom-associative colour algebra $(\mathcal{A},\mu,\varepsilon,\beta,\alpha)$ is BiHom-Lie admissible.
\end{prop}
\begin{proof}The $\varepsilon$-skew symmetry follows straightaway from the definition. Assume that $G$ is a
subgroup of $S_3$. Then $S_3$ can be written as the disjoint union of the left cosets of $G$. Say
$S_3 = \bigcup\limits_{\sigma \in I} \sigma G$, with $I \subseteq S_3$ and for any $\sigma, \sigma^{'} \in I,$
\begin{eqnarray*}
&& \sigma \neq \sigma^{'}\Longrightarrow \sigma G\bigcap \sigma^{'} G=\varnothing.
\end{eqnarray*}
Then one has
\begin{eqnarray*}
  && \sum_{\sigma \in S_3}(-1)^{sgn(\sigma)}(-1)^{|\sigma(x_1, x_2, x_3)|}as_{\alpha,\beta}\circ \sigma(\alpha^{-1}\beta^2(x_1),\beta(x_2),\alpha(x_3))\\
  && =\sum_{\tau \in I}\sum_{\sigma \in \tau G}(-1)^{sgn(\sigma)}(-1)^{|\sigma(x_1, x_2, x_3)|}as_{\alpha,\beta}\circ \sigma(\alpha^{-1}\beta^2(x_1),\beta(x_2),\alpha(x_3))=0.
\end{eqnarray*}
\end{proof}

Now we provide a classification of  BiHom-Lie colour admissible algebras via $G$-BiHom-associative colour algebras. The subgroups of $S_3$ are:
\begin{eqnarray*}
&& G_1=\{Id\},~~G_2=\{Id,\sigma_1\},~~G_3=\{Id,\sigma_2\},~~G_4=\{Id,\sigma_2\sigma_1\sigma_2=(1~~3)\},~~G_5=A_3,~~G_6=S_3,
\end{eqnarray*}
where $A_3$ is the alternating subgroup of $S_3$.\\
We obtain the following types of BiHom-Lie admissible colour algebras.
\begin{enumerate}
\item[]$\bullet$ The $G_1$-BiHom-associative colour algebras are the colour BiHom-algebras defined in Definition (\ref{assoc}).\\
\item[]$\bullet$ The $G_2$-BiHom-associative colour algebras satisfy the condition:
\begin{eqnarray*}
&&\beta^2(x)(\beta(y)\alpha(z))-(\alpha^{-1}\beta^2(x)\beta(y))\alpha(\beta(z))\\
&&=\varepsilon(x,y)\Big(\alpha(\beta(y))(\alpha^{-1}\beta^2(x)\alpha(z))
-(\beta(y)\alpha^{-1}\beta^2(x))\alpha(\beta(z)) \Big).
\end{eqnarray*}
\item[]$\bullet$ The $G_3$-BiHom-associative colour algebras satisfy the condition:
\begin{eqnarray*}
 \beta^2(x)(\beta(y)\alpha(z))-(\alpha^{-1}\beta^2(x)\beta(y))\alpha\beta(z) =\varepsilon(y,z)\Big(\beta^2(x)(\beta(z)\alpha(y))-(\alpha^{-1}\beta^2(x)\alpha(z))\beta^2(y) \Big).
\end{eqnarray*}
\item[]$\bullet$ The $G_4$-BiHom-associative colour algebras satisfy the condition:
\begin{eqnarray*}
&&\beta^2(x)(\beta(y)\alpha(z))-(\alpha^{-1}\beta^2(x)\beta(y))\alpha\beta(z) \\ &=&\varepsilon(x,y)\varepsilon(y,z)\varepsilon(x,z)\Big(\alpha^2(z)(\beta(y)\alpha^{-1}\beta^2(x))-(\alpha(z)\beta(y))\alpha^{-1}\beta^3(x) \Big).
\end{eqnarray*}
\item[]$\bullet$ The $G_5$-BiHom-associative colour algebras satisfy the condition:
\begin{eqnarray*}
 && \beta^2(x)(\beta(y)\alpha(z))-\varepsilon(x,y+z)\alpha\beta(y)(\alpha(z)\alpha^{-1}\beta^2(x))-\varepsilon(x+y,z)\alpha^2(z)(\alpha^{-1}\beta^2(x)\beta(y))\\ &&=(\alpha^{-1}\beta^2(x)\beta(y))\alpha\beta(z)-\varepsilon(x,y+z)(\beta(y)\alpha(z))\alpha^{-1}\beta^3(x)-\varepsilon(x+y,z)(\alpha(z)\alpha^{-1}\beta^2(x))\beta^2(y).
\end{eqnarray*}
\item[]$\bullet$ The $G_6$-BiHom-associative colour algebras are the BiHom-Lie colour admissible algebras.
\end{enumerate}
\begin{rem}
Moreover, if in the previous identities we consider $\beta= \alpha$, then we obtain a classification
of Hom-Lie admissible colour algebras \cite{yuan2010hom}.
\end{rem}

\section{Cohomology and Representations of BiHom-Lie colour algebras}

\subsection{Cohomology of BiHom-Lie colour algebras}
%
%
%

We extend first to  BiHom-Lie colour algebras, the concept of $\A$-module  introduced in \cite{sadaoui2011cohomology,benayadi2010hom,sheng2010representations}, and then define a family of cohomology complexes for BiHom-Lie colour  algebras.\\

\begin{df}Let $(\mathcal{A},[\cdot,\cdot],\varepsilon,\alpha,\beta)$ be a regular BiHom-Lie colour algebra. A representation of $\mathcal{A}$ is a $4$-tuple $(V,\rho,\alpha_V,\beta_V)$, where $V$ is a $\Gamma$-graded vector space, $\alpha_V,\beta_V:V \longrightarrow V$ are two even commuting linear maps and $\rho:\mathcal{A}\longrightarrow End(V)$ is an even linear map such that, for all $x,y \in \mathcal{H}(\mathcal{A})$ and $v\in V$, we have
\begin{eqnarray}
  \rho(\alpha(x))\circ \alpha_V &=& \alpha_V \circ \rho(x), \nonumber\\
  \rho(\beta(x))\circ \beta_V &=& \beta_V \circ \rho(x),\nonumber \\
  \rho([\beta(x),y])\circ \beta_V(v)&=& \rho(\alpha\beta(x))\circ \rho(y)(v)-\varepsilon(x,y)\rho(\beta(y))\circ \rho(\alpha(x))(v).\label{3.39}
\end{eqnarray}
\end{df}

The cohomology of Lie colour algebras was introduced in \cite{scheunert1979generalized}. In the following, we  define cohomology complexes of  BiHom-Lie colour   algebras.\\

Let $(\mathcal{A},[\cdot,\cdot],\varepsilon,\alpha,\beta)$ be a regular BiHom-Lie colour algebra and $(V,\rho,\alpha_V,\beta_V)$ be a representation of $\mathcal{A}$. In the sequel, we denote $\rho(x)(v)$ by a bracket $[x,v]_V$.

The set of $n$-cochains on $\mathcal{A}$ with
values in $V$, which we denote by $C^n(\mathcal{A},V)$, is the set of skewsymmetric $n$-linear maps $f:\mathcal{A}^{n}\rightarrow V $, that is
$$f(x_{1},...,x_{i},x_{i+1},...,x_{n})=-\varepsilon(x_{i},x_{i+1})f(x_{1},...,x_{i+1},x_{i},...,x_{n}),~~\forall~~ 1\leq i \leq n-1.$$
For $n=0$, we have $C^{0}(\mathcal{A},V)=V$.

We set
\begin{eqnarray*}
 C_{\alpha,\beta}^n(\mathcal{A},V) &=& \{f:\mathcal{A}^{n}\longrightarrow V : \ f\in C^n(\mathcal{A},V) \text{ and } f \circ \alpha=\alpha_V \circ f,~~f \circ \beta=\beta_V \circ f \}.
\end{eqnarray*}


We extend this definition to the case of integers $n < 0$  and set
$$C_{\alpha,\beta}^{n}(\mathcal{A},V)=\{0\}\ \text{ if } n < -1\quad \text{ and }\quad
C_{\alpha,\beta}^{0}(\mathcal{A},V)=V.$$
A map $f\in C^n(\mathcal{A},V)$ is called even (resp. of degree  $\gamma$) when $f(x_{1},...,x_{i},...,x_{n}) \in V_{\gamma_{1}+...+\gamma_{i}+...+\gamma_{n}}$ for all elements $x_i\in \mathcal{A}_{\gamma_{i}}$
 (resp. $f(x_{1},...,x_{i},...,x_{n}) \in V_{\gamma+\gamma_{1}+...+\gamma_{i}+...+\gamma_{n}}$).

A homogeneous element $f \in C_{\alpha,\beta}^{n}(\mathcal{A},V)$ is called $n$-cochain or sometimes in the literature $n$-Hom-cochain.\\
Next, for a given  integer $r$, we define the coboundary operator $\delta_{r}^{n}$.
\begin{df} We call, for $n\geq 1$ and for  any integer $r$, a $n$-coboundary operator of the  BiHom-Lie colour algebra $(\mathcal{A},[.,.],\varepsilon,\alpha,\beta)$ the linear map $\delta_{r}^{n}:C_{\alpha,\beta}^{n}(\mathcal{A},V)\longrightarrow C_{\alpha,\beta}^{n+1}(\mathcal{A},V)$ defined by
\begin{eqnarray}\label{ECHLC}
&& \delta_{r}^{n}(f)(x_{0},....,x_{n}) = \\
&&\quad  \nonumber \sum \limits_{0\leq s < t \leq n}(-1)^{t}\varepsilon(x_{s+1}+...+x_{t-1},x_{t})
 f(\beta(x_{0}),...,\beta(x_{s-1}),[\alpha^{-1}\beta(x_{s}),x_{t}],\beta(x_{s+1}),...,\widehat{x_{t}},...,\beta(x_{n}))\\
&& \nonumber \quad + \sum \limits_{s=0}^{n}(-1)^{s}\varepsilon(\gamma+x_{0}+...+x_{s-1},x_{s})[\alpha\beta^{r+n-1}(x_{s}),f(x_{0},...,\widehat{x_{s}},..,x_{n})]_{V},
\end{eqnarray}
where $f \in C_{\alpha,\beta}^{n}(\mathcal{A},V)$, $\gamma$ is the degree of $f$, $(x_{0}, ...., x_{n}) \in \mathcal{H}(\mathcal{A})^{\otimes n+1} $ and  $\widehat{x}$ indicates that the element $x$ is omitted.\\
In the sequel we assume that the BiHom-Lie colour algebra $(\mathcal{A},[.,.],\varepsilon,\alpha,\beta)$ is multiplicative.
\end{df}
For $n=1$, we have
$$\begin{array}{cccc}
  \delta_{r}^{1}: & C^{1}(\mathcal{A},V) & \longrightarrow & C^{2}(\mathcal{A},V)  \\
   & f & \longmapsto & \delta_{r}^{1}(f)
\end{array}$$
such that for two homogeneous elements $x,y$ in $\mathcal{A}$
\begin{equation}\label{cobord1}
    \delta_{r}^{1}(f)(x,y) =\varepsilon(\gamma,x)[\alpha\beta(x),f(y)]_V-\varepsilon(\gamma+x,y)[\alpha\beta(y),f(x)]_V-f([\alpha^{-1}\beta(x),y])
\end{equation}
and for $n=2$, we have
$$\begin{array}{cccc}
  \delta_{r}^{2}: & C^{2}(\mathcal{A},V) & \longrightarrow & C^{3}(\mathcal{A},V)  \\
   & f & \longmapsto & \delta_{r}^{2}(f)
\end{array}$$
such that, for three homogeneous elements $x,y,z$ in $\mathcal{A}$, we have
\begin{eqnarray}\label{2-cocycle}
\delta_{r}^2(f)(x,y,z)&=&\varepsilon(\gamma,x)[\alpha\beta^2(x),f(y,z)]_V-\varepsilon(\gamma+x,y)[\alpha\beta^2(y),f(x,z)]_V \nonumber\\
&+&\varepsilon(\gamma+x+y,z)[\alpha\beta^2(z),f(x,y)]_V-f([\alpha^{-1}\beta(x),y],\beta(z))\nonumber\\
&+&\varepsilon(y,z)f([\alpha^{-1}\beta(x),z],\beta(y))+f(\beta(x),[\alpha^{-1}\beta(y),z]).
\end{eqnarray}

\begin{lem} With the above notations, for any $f\in C^{n}_{\alpha,\beta}(\mathcal{A},V)$, we have
\begin{eqnarray*}
\delta_{r}^{n}(f)\circ \alpha &=& \alpha_V \circ \delta_{r}^{n}(f),\\
\delta_{r}^{n}(f)\circ \beta &=& \beta_V \circ \delta_{r}^{n}(f),~~\forall~~ n \geq 2
\end{eqnarray*}
Thus, we obtain a well defined map $\delta_{r}^{n}:C_{\alpha,\beta}^{n}(\mathcal{A},V)\longrightarrow C_{\alpha,\beta}^{n+1}(\mathcal{A},V).$
\end{lem}
\begin{proof} Let $f\in C^{n}_{\alpha,\beta}(\mathcal{A},V)$ and $(x_{0},....,x_{n})\in \mathcal{H}(\mathcal{A})^{\otimes n+1}$, we have
\begin{small}
\begin{eqnarray*}
&& \delta_{r}^{n}(f)\circ \alpha(x_{0},....,x_{n})\\
&&=\small \small \delta_{r}^{n}(f)(\alpha(x_{0}),....,\alpha(x_{n}))\\
&&=\small \sum \limits_{0\leq s < t \leq n}(-1)^{t}\varepsilon(x_{0}+...+x_{t-1},x_{t})f(\alpha\beta(x_{0}),...,
\alpha\beta(x_{s-1}), [\alpha\alpha^{-1}\beta(x_{s}),\alpha(x_{t})],\alpha\beta(x_{s+1}),...,\widehat{x_{t}},...,\alpha\beta(x_{n}))\\
&&+\sum \limits_{s=0}^{n}(-1)^{s}\varepsilon(\gamma+x_{0}+...+x_{s-1},x_{s})
[\alpha^2 \beta^{n-1+r}(x_{s}),f(\alpha(x_{0}),...,\widehat{x_{s}},..,\alpha(x_{n}))]_{V}\\
&&=\small \sum \limits_{0\leq s < t \leq n}(-1)^{t}\varepsilon(x_{0}+...+x_{t-1},x_{t})\alpha_V \circ f(\beta(x_{0}),...,
\beta(x_{s-1}),[\alpha^{-1}\beta(x_{s}),x_{t}],\beta(x_{s+1}),...,\widehat{x_{t}},...,\beta(x_{n}))\\
&&+\sum \limits_{s=0}^{n}(-1)^{s}\varepsilon(\gamma+x_{0}+...+x_{s-1},x_{s})
\alpha_V[\alpha \beta^{n-1+r}(x_{s}),f(x_{0},...,\widehat{x_{s}},..,x_{n})]_{V}\\
&& =\alpha_V \circ \delta_{r}^{n}(f)(x_{0},....,x_{n}).
\end{eqnarray*}
\end{small}
Then $\delta_{r}^{n}(f)\circ \alpha=\alpha_V \circ \delta_{r}^{n}(f)$.\\
Similarly, we have
\begin{eqnarray*}
 \delta_{r}^{n}(f)\circ \beta(x_{0},....,x_{n})
&=& \beta_V \circ \delta_{r}^{n}(f)(x_{0},....,x_{n}),
\end{eqnarray*}
 which completes the proof.
\end{proof}
\begin{thm}\label{ker} Let $(\mathcal{A},[.,.],\varepsilon,\alpha,\beta)$ be a multiplicative BiHom-Lie colour algebra and $(V,\alpha_V,\beta_V)$ be an $\mathcal{A}$-module.
Then the pair $(\bigoplus\limits_{n\geq 0}C_{\alpha,\beta}^{n}, \delta_{r}^{n})$ is a cohomology complex. That is  the maps $\delta_{r}^{n}$ satisfy  $\delta_{r}^{n}\circ \delta_{r}^{n-1}=0,~~\forall~~ n \geq 2,\forall~~ r \geq 1.$
\end{thm}
 \begin{proof} For any $f \in C^{n-1}(\mathcal{A},V)$, we have
 \begin{small}
 \begin{eqnarray}\label{1}
  && \delta_{r}^{n}\circ \delta_{r}^{n-1}(f)(x_{0},\dots,x_{n})\nonumber\\
&=& \quad\quad \sum_{s < t}(-1)^{t}\varepsilon(x_{0}+\dots+x_{t-1},x_{t})
 \delta_{r}^{n-1}(f)(\beta(x_{0}),\dots,\beta(x_{s-1}),[\alpha^{-1}\beta(x_{s}),x_{t}],\beta(x_{s+1}),\dots,
\widehat{x_{t}},\dots,\beta(x_{n}))\nonumber \\
\end{eqnarray}
\begin{eqnarray}\label{2}
     &+&\sum_{s=0}^{n}(-1)^{s}\varepsilon(f+x_{0}+...+x_{s-1},x_{s})[\alpha\beta^{r+n-1}(x_{s}),
 \delta_{r}^{n-1}(f)(x_{0},...,\widehat{x_{s}},..,x_{n})]_{V}.
\end{eqnarray}
From $(\ref{1})$ we have
\begin{small}
\begin{eqnarray*}
&&\delta_{r}^{n-1}(f)(\beta(x_{0}),\dots,\beta(x_{s-1}),[\alpha^{-1}\beta(x_{s}),x_{t}],\beta(x_{s+1}),\dots,
\widehat{x_{t}},\dots,\beta(x_{n}))\\
&=&\sum \limits_{s^{'} < t^{'}< s }(-1)^{t^{'}}\varepsilon(x_{s^{'}+1}+...+x_{t^{'}-1},x_{t^{'}})f\Big(\beta^{2}(x_{0})
,...,\beta^{2}(x_{s^{'}-1}),[\alpha^{-1}\beta\beta(x_{s^{'}}),\beta(x_{t^{'}})],\beta^{2}(x_{s^{'}+1}),
\end{eqnarray*}
 \begin{equation}\label{01}
...,\widehat{x_{t^{'}}},...,\beta^{2}(x_{s-1}),\beta([\alpha^{-1}\beta(x_{s}),x_{t}]),\beta^{2}(x_{s+1}),...,\widehat{x_{t}},...,\beta^{2}(x_{n}) \Big)
\end{equation}
\begin{eqnarray*}
&+&\sum \limits_{s^{'}< s }(-1)^{s}\varepsilon(x_{s^{'}+1}+...+x_{s-1},x_{s})
 \end{eqnarray*}
 \begin{eqnarray}\label{02}
&& f\Big(\beta^{2}(x_{0}),...,\beta^{2}(x_{s^{'}-1}),[\beta(x_{s^{'}-1}),[\alpha^{-2}\beta^2(x_{s}),\alpha^{-1}\beta(x_{t})]],\beta^{2}(x_{s^{'}+1}),...,
\widehat{x_{s,t}},...,\beta^{2}(x_{n}) \Big)
\end{eqnarray}
\begin{eqnarray*}
&+&\sum \limits_{s^{'} < s< t^{'}< t }(-1)^{t^{'}}\varepsilon(x_{s^{'}+1}+...+[x_{s},x_{t}]+...+x_{t^{'}-1},x_{t^{'}})
 \end{eqnarray*}
 \begin{equation}\label{03}
f \Big(\beta^{2}(x_{0}),...,\beta^{2}(x_{s^{'}-1}),[\alpha^{-1}\beta\beta(x_{s^{'}}),\beta(x_{t^{'}})],\beta^{2}(x_{s^{'}+1}),\beta([\alpha^{-1}\beta(x_{s}),x_{t}])
,...,\widehat{x_{t^{'}}},...,\beta^{2}(x_{n}) \Big)
\end{equation}
\begin{eqnarray*}
&+&\sum \limits_{s^{'}< s< t< t^{'}}
(-1)^{t^{'}}\varepsilon(x_{s^{'}+1}+...x_{s-1}+[x_{s},x_{t}]+x_{s+1}+...+\widehat{x_{t}}+...+x_{t^{'}-1},x_{t^{'}})\end{eqnarray*}
\begin{equation}\label{04}
 f \Big(\beta^{2}(x_{0}),...,\beta^{2}(x_{s^{'}-1}),[\alpha^{-1}\beta\beta(x_{s^{'}}),\beta(x_{t^{'}})],\beta^{2}(x_{s^{'}+1})
,\beta([\alpha^{-1}\beta(x_{s}),x_{t}]),...,\widehat{x_{t}},...,\widehat{x_{t^{'}}},...,\beta^{2}(x_{n}) \Big)
\end{equation}
\begin{small}
\begin{eqnarray}\label{05}
+\sum \limits_{ s< t^{'}< t}
(-1)^{t^{'}}\varepsilon(x_{s+1}+...+x_{t^{'}-1},x_{t^{'}}) f \Big(\beta^{2}(x_{0}),...,[[\alpha^{-2}\beta^2(x_{s}),\alpha^{-1}\beta(x_{t})],\beta(x_{t^{'}})],
\beta^{2}(x_{s+1}),...,\widehat{x_{t,t^{'}}},...,\beta^{2}(x_{n}) \Big)
\end{eqnarray}
\end{small}
\begin{eqnarray*}
&+&\sum \limits_{ s<t <t^{'}}(-1)^{t^{'}-1}\varepsilon(x_{s+1}+...+\widehat{x_{t}}+...+x_{t^{'}-1},x_{t^{'}})\end{eqnarray*}
\begin{eqnarray}\label{06}
f\Big(\beta^{2}(x_{0}),...,\beta^{2}(x_{s-1}),[[\alpha^{-2}\beta^2(x_{s}),\alpha^{-1}\beta(x_{t})],\beta(x_{t^{'}})],\beta^{2}(x_{s+1}),...,\widehat{x_{t,t^{'}}}
,...,\beta^{2}(x_{n}) \Big)
\end{eqnarray}
\begin{eqnarray*}
&+&\sum \limits_{ s<s^{'}< t^{'}< t}(-1)^{t^{'}}\varepsilon(x_{s^{'}+1}+...+x_{t^{'}-1},x_{t^{'}})
f\Big(\beta^{2}(x_{0}),...,\beta^{2}(x_{s-1}),\beta([\alpha^{-1}\beta(x_{s}),x_{t}]),\end{eqnarray*}
\begin{eqnarray}\label{07}
[\alpha^{-1}\beta\beta(x_{s^{'}}),\beta(x_{t^{'}})],...,\widehat{x_{t^{'}}},...,\widehat{x_{t}},...,\beta^{2}(x_{n}) \Big)
\end{eqnarray}
\begin{eqnarray*}
&+&\sum \limits_{ s<s^{'}< t<t^{'}}(-1)^{t^{'}}\varepsilon(x_{s^{'}+1}+...+\widehat{x_{t}}+...+x_{t^{'}-1},x_{t^{'}})
\end{eqnarray*}
\begin{eqnarray}\label{08}
f \Big(\beta^{2}(x_{0}),...,\beta^{2}(x_{s-1}),\beta([\alpha^{-1}\beta(x_{s}),x_{t}]),\beta^{2}(x_{s+1}),..., [\alpha^{-1}\beta\beta((x_{s^{'}}),\beta(x_{t^{'}})],...,\widehat{x_{t}},...,\widehat{x_{t^{'}}},...,\beta^{2}(x_{n}) \Big)
\end{eqnarray}
\begin{eqnarray*}
&+&\sum_{ t<s^{'}<t^{'}}(-1)^{t^{'}}\varepsilon(x_{s^{'}+1}+...+\widehat{x_{t,t^{'}}}+...+x_{t^{'}-1},x_{t^{'}})
\end{eqnarray*}
\begin{eqnarray}\label{O9}
f \Big(\beta^{2}(x_{0}),...,\beta^{2}(x_{s-1}),\beta([\alpha^{-1}\beta(x_{s}),x_{t}]),\beta^{2}(x_{s+1}),...,\widehat{x_{t}},..., [\alpha^{-1}\beta\beta(x_{s^{'}}),\beta(x_{t^{'}})],...,\widehat{x_{t^{'}}},...,\beta^{2}(x_{n}) \Big)
\end{eqnarray}
\begin{eqnarray}\label{010}
&+&\sum \limits_{ 0< s^{'}< s}(-1)^{s^{'}}\varepsilon(\gamma+x_{0}+...+x_{s^{'}-1},x_{s^{'}}) [\alpha^{r+n-1}(x_{s^{'}}),f(\beta(x_{0}),...\widehat{x_{s^{'}}}
,[x_{s},x_{t}],...,\widehat{x_{t^{'}}},...,\beta(x_{n}))]_{V}\nonumber \\
\end{eqnarray}
\begin{eqnarray*}
&+&(-1)^{s}\varepsilon(\gamma+x_{0}+...+x_{s-1},[x_{s},x_{t}])
\end{eqnarray*}
\begin{eqnarray}\label{011}
[\alpha\beta^{r+n-3}([\alpha^{-1}\beta(x_{s}),x_{t}]),
f(\beta(x_{0}),...,\widehat{[x_{s},x_{t}]},\beta(x_{s+1}),...,\widehat{x_{t}},...,\beta(x_{n}))]_{V}
\end{eqnarray}
\begin{eqnarray*}
& +&\sum \limits_{ s< s^{'}< t}(-1)^{s^{'}}\varepsilon(\gamma+x_{0}+...+[x_{s},x_{t}]+...+x_{s^{'}-1},x_{s^{'}})\end{eqnarray*}
\begin{eqnarray}\label{012}
&&[\alpha\beta^{r+n-2}(x_{s^{'}}),f(\beta(x_{0}),...,[\alpha^{-1}\beta(x_{s}),x_{t}],...,\widehat{x_{s^{'},t}},...,\beta(x_{n}))]_{V}
\end{eqnarray}
\begin{eqnarray*}
&+&\sum \limits_{ t< s^{'}}(-1)^{s^{'}}\varepsilon(\gamma+x_{0}+..[x_{s},x_{t}]+...+\widehat{x_{t}}+...+x_{s^{'}-1},x_{s^{'}})\end{eqnarray*}
\begin{eqnarray}\label{013}
&&[\alpha\beta^{r+n-2}(x_{s^{'}}),f(\beta(x_{0}),...,[\alpha^{-1}\beta(x_{s}),x_{t}],...,\widehat{x_{t,s^{'}}},...,\beta(x_{n}))]_{V}.
\end{eqnarray}
The identity $(\ref{2})$ implies that
\begin{eqnarray*}[\alpha\beta^{r+n-1}(x_{s}), \delta_{r}^{n-1}(f)(x_{0},...,\widehat{x_{s}},..,x_{n})]_{V}
&=&[\alpha\beta^{r+n-1}(x_{s}),\sum \limits_{s^{'} < t^{'}< s }(-1)^{t^{'}}\varepsilon(x_{s^{'}+1}+...+x_{t^{'}-1},x_{t^{'}})
\end{eqnarray*}
\begin{eqnarray*}
&& f \Big(\beta(x_{0}),...,\beta(x_{s^{'}-1}),[\alpha^{-1}\beta(x_{s^{'}}),x_{t^{'}}],\beta(x_{s^{'}+1}),...,\widehat{x_{s^{'},t^{'},t}},\beta(x_{s+1}),...,\beta(x_{n}) \Big)]_{V}
\end{eqnarray*}
\begin{eqnarray*}
&+&[\alpha\beta^{r+n-1}(x_{s}),\sum_{s^{'} < s<t }(-1)^{t^{'}-1}\varepsilon(x_{s^{'}+1}+...+\widehat{x_{s}}+...+x_{t^{'}-1},x_{t^{'}})\end{eqnarray*}
\begin{equation}\label{015}
f(\beta(x_{0}),...,\beta(x_{s^{'}-1}),[\alpha^{-1}\beta(x_{s^{'}}),x_{t^{'}}],\beta(x_{s^{'}+1}),...,\widehat{x_{t,s^{'}}},...,\beta(x_{n}))]_{V}
\end{equation}
\begin{eqnarray*}&+&[\alpha\beta^{r+n-1}(x_{s}),\sum \limits_{s < s^{'}<t^{'} }(-1)^{t^{'}}\varepsilon(x_{s^{'}+1}+...+x_{t^{'}-1},x_{t^{'}})\end{eqnarray*}
\begin{equation}\label{016}
f(\beta(x_{0}),...,\widehat{x_{s}},...,\beta(x_{s^{'}-1}),[\alpha^{-1}\beta(x_{s^{'}}),x_{t^{'}}],\beta(x_{s^{'}+1}),...,\widehat{x_{t^{'}}},...,\beta(x_{n}))]_{V}
\end{equation}
\begin{equation}\label{017}
+[\alpha\beta^{r+n-1}(x_{s}),\sum_{s^{'}=0}^{s-1}(-1)^{s}\varepsilon(\gamma+x_{0}+...+x_{s^{'}-1},x_{s^{'}})[\alpha\beta^{r+n-2}(x_{s^{'}}),
f(x_{0},...,\widehat{x_{s^{'},s}},...,x_{n})]_{V}]_{V}
\end{equation}

\begin{align}
&+[\alpha\beta^{r+n-1}(x_{s}),\sum_{s^{'}=s+1}^{n}(-1)^{s^{'}-1}\varepsilon(\gamma+x_{0}+...+\widehat{x_{s}}+...+x_{s^{'}-1},x_{s^{'}})
[\alpha\beta^{n+r-2}(x_{s^{'}}),f(x_{0},...,\widehat{x_{s^{'},s}},...,x_{n})]_{V}]_{V}.
\nonumber \\ \label{018} &&
\end{align}
\end{small}
By the $\varepsilon$-Bihom-Jacobi condition, we obtain
$$\sum \limits_{ s<t }(-1)^{t}\varepsilon(x_{s+1}+...+x_{t-1},x_{t})((\ref{02})+(\ref{05})+(\ref{06}))=0.$$
Also, we have
\begin{eqnarray*}
  (\ref{011}) &=&[\alpha\beta^{r+n-3}([\alpha^{-1}\beta(x_{s}),x_{t}]),
f(\beta(x_{0}),...,\widehat{[x_{s},x_{t}]},\beta(x_{s+1}),...,\widehat{x_{t}},...,\beta(x_{n}))]_{V} \\
   &=& [\alpha\beta^{n+r-1}(x_{s}),[\alpha\beta^{r+n-2}(x_{t}),f(x_{0},...,\widehat{x_{s,t}},...,x_{n})]_{V}]_{V}\\
   &-&[\alpha\beta^{r+n-1}(x_{t}),[\alpha\beta^{r+n-2}(x_{s}),f(x_{0},...,\widehat{x_{s,t}},...,x_{n})]_{V}]_{V}.
\end{eqnarray*}
Thus
\begin{eqnarray*}
&&\ \ \ \sum \limits_{ s < t }(-1)^{t}\varepsilon(x_{s+1}+...+x_{t-1},x_{t})(\ref{011})
+\sum \limits_{s=0}^{n}(-1)^{s}\varepsilon(\gamma+x_{0}+...+x_{s-1},x_{s})(\ref{017})\\
&&\ \ \ +\sum \limits_{s=0}^{n}(-1)^{s}\varepsilon(\gamma+x_{0}+...+x_{s-1},x_{s})(\ref{018})=0.
\end{eqnarray*}
By a simple calculation, we get
\begin{eqnarray*}
\sum \limits_{ s < t }(-1)^{t}\varepsilon(x_{s+1}+...+x_{t-1},x_{t})(\ref{010})
 +\sum \limits_{s=0}^{n}(-1)^{s}\varepsilon(\gamma+x_{0}+...+x_{s-1},x_{s})(\ref{016})&=&0,\\
 \sum \limits_{ s < t }(-1)^{t}\varepsilon(x_{s+1}+...+x_{t-1},x_{t})(\ref{013})
+\sum \limits_{s=0}^{n}(-1)^{s}\varepsilon(\gamma+x_{0}+...+x_{s-1},x_{s})(\ref{015})&=&0,
\end{eqnarray*}
and
\begin{eqnarray*}&&\ \ \ \sum \limits_{ s < t }(-1)^{t}\varepsilon(x_{s+1}+...+x_{t-1},x_{t})((\ref{03})+(\ref{08}))\\
&&\ \ \ =\sum \limits_{ s < t }(-1)^{t}\varepsilon(x_{s+1}+...+x_{t-1},x_{t})\Big(\sum \limits_{s^{'} < s< t^{'}< t} (-1)^{t^{'}} \varepsilon(x_{s^{'}+1}+...+[x_{s},x_{t}]+...+x_{t^{'}-1},x_{t^{'}})\\
&&\ \ \ \quad f(\alpha^{2}(x_{0}),..., \alpha^{2}(x_{s^{'}-1}),[\alpha(x_{s^{'}}),\alpha(x_{t^{'}})] ,\alpha^{2}(x_{s^{'}+1})
,\alpha([x_{s},x_{t}]),..., \widehat{x_{t^{'}}},...,\alpha^{2}(x_{n}))\Big)\\
&&\ \ \ +\sum \limits_{ s < t }(-1)^{t}\varepsilon(x_{s+1}+...+x_{t-1},x_{t})\Big( \sum_{ s<s^{'}< t<t^{'}}
(-1)^{t^{'}}\varepsilon(x_{s^{'}+1}+...+\widehat{x_{t}}+...+x_{t^{'}-1},x_{t^{'}})\\
&&\ \ \ \quad f(\alpha^{2}(x_{0}),...\alpha^{2}(x_{s-1}),\alpha([x_{s},x_{t}]), \alpha^{2}(x_{s+1}),...,\widehat{x_{t,t^{'}}},...,[\alpha(x_{s^{'}}),\alpha(x_{t^{'}})],...,\alpha^{2}(x_{n}))\Big)\\
&&\ \ \ =0.\end{eqnarray*}\end{small}
Similarly, we have
$$\sum \limits_{ s < t }(-1)^{t}\varepsilon(x_{s+1}+...+x_{t-1},x_{t})((\ref{01})+(\ref{O9}))=0$$
and
 $$\sum \limits_{ s < t }(-1)^{t}\varepsilon(x_{s+1}+...+x_{t-1},x_{t})((\ref{04})+(\ref{07}))=0.$$
Therefore  $\delta_{r}^{n}\circ \delta_{r}^{n-1}=0$, which completes the proof.
\end{proof}

Let $Z_{r}^{n}(\mathcal{A},V)$ (resp. $B_{r}^{n}(\mathcal{A},V)$) denote the kernel of $\delta_{r}^{n}$ (resp. the image of $\delta_{r}^{n-1}$). The spaces $Z_{r}^{n}(\mathcal{A},V)$ and $B_{r}^{n}(\mathcal{A},V)$ are graded submodules of $C_{\alpha,\beta}^{n}(\mathcal{A},V)$ and according to Proposition \ref{ker}, we have
\begin{equation}\label{BZ}
    B_{r}^{n}(\mathcal{A},V)\subseteq Z_{r}^{n}(\mathcal{A},V).
\end{equation}
The elements of $Z_{r}^{n}(\mathcal{A},V)$ are called $n$-cocycles, and  the elements of $B_{r}^{n}(\mathcal{A},V)$ are called the $n$-coboundaries. Thus, we define a so-called cohomology groups
$$H_{r}^{n}(\mathcal{A},V)=\frac{Z_{r}^{n}(\mathcal{A},V)}{B_{r}^{n}(\mathcal{A},V)}.$$
We denote by $H_{r}^{n}(\mathcal{A},V)=\bigoplus_{\gamma \in \Gamma}(H_{r}^{n}(\mathcal{A},V))_{\gamma}$ the space of all $r$-cohomology group of degree $\gamma$ of the  BiHom-Lie colour algebra  $\mathcal{A}$ with values in $V$.\\
Two elements of $Z_{r}^{n}(\mathcal{A},V)$ are said to be cohomologous if their residue classes modulo $B_{r}^{n}(\mathcal{A},V)$ coincide, that is if their difference lies in $B_{r}^{n}(\mathcal{A},V)$.

\subsection{Adjoint representations of  BiHom-Lie colour algebras}
In this section, we generalize to BiHom-Lie colour algebras some results from \cite{sadaoui2011cohomology} and \cite{sheng2010representations}. Let $(\mathcal{A},[.,.],\varepsilon,\alpha,\beta)$ be a regular BiHom-Lie colour algebra. We consider that $\mathcal{A}$ represents on itself
via the bracket with respect to the morphisms $\alpha,\beta$.

Now, we discuss  adjoint representations of a  BiHom-Lie colour algebra.

The adjoint representations are generalized in the following way.
\begin{df} An $\alpha^{s}\beta^{l}$-adjoint representation, denoted  by $ad_{s,l}$, of a  BiHom-Lie colour algebra $(\mathcal{A},[.,.],\varepsilon,\alpha,\beta)$  is defined  as
$$ad_{s,l}(a)(x)=[\alpha^{s}\beta^l(a),x],~~\forall~~ a,x \in \mathcal{H}(\mathcal{A}).$$
\end{df}
\begin{lem} With the above notations, we have  $(\mathcal{A},ad_{s,l}(.)(.),\alpha,\beta)$ is a representation of the  BiHom-Lie colour algebra $(\mathcal{A},[.,.],\varepsilon,\alpha,\beta)$. It satisfies
\begin{eqnarray*}
ad_{s,l}(\alpha(x))\circ \alpha &=& \alpha\circ ad_{s,l}(x),\\
ad_{s,l}(\beta(x))\circ \beta &=& \beta\circ ad_{s,l}(x),\\
ad_{s,l}([\beta(x),y])\circ \beta &=& ad_{s,l}(\alpha \beta(x))\circ ad_{s,l}(y)-\varepsilon(x,y)ad_{s,l}(\beta(y))\circ ad_{s,l}(x).
\end{eqnarray*}
\end{lem}
\begin{proof} First, the result follows from
\begin{eqnarray*}
  ad_{s,l}(\alpha(x))(\alpha(y)) =[\alpha^{s}\beta^l(\alpha(x)),\alpha(y)]
   = \alpha([\alpha^{s}\beta^l(x),y])
   =  \alpha\circ ad_{s,l}(x)(y).
\end{eqnarray*}
Similarly, we have
\begin{eqnarray*}
  ad_{s,l}(\beta(x))(\beta(y)) &=&  \beta \circ ad_{s,l}(x)(y).
\end{eqnarray*}
Note that the $\varepsilon$-BiHom skew symmetry condition implies
 \begin{eqnarray*}
  ad_{s,l}(x)(y) &=&  -\varepsilon(x,y)[\alpha^{-1}\beta(y),\alpha^{s+1}\beta^{l-1}(x)].
\end{eqnarray*}
On one hand, we have
\begin{small}
\begin{eqnarray*}
  ad_{s,l}([\beta(x),y])(\beta(z)) &=& -\varepsilon(x+y,z)[\alpha^{-1}\beta^2(z),\alpha^{s+1}\beta^{l-1}([\beta(x),y])]\\
   &=&-\varepsilon(x+y,z)[\alpha^{-1}\beta^2(z),[\alpha^{s+1}\beta^l(x),\alpha^{s+1}\beta^{l-1}(y)]].
\end{eqnarray*}
\end{small}
On the other hand, we have
\begin{small}
\begin{eqnarray*}
  &&\Big(ad_{s,l}(\alpha\beta(x))\circ ad_{s,l}(y)\Big)(z)- \varepsilon(x,y)\Big( ad_{s,l}(\beta(y))\circ ad_{s,l}(\alpha(x))\Big)(z)\\
  &=& ad_{s,l}(\alpha\beta(x))(-[\alpha^{-1}\beta(z),\alpha^{s+2}\beta^{l-1}(y)])
  -\varepsilon(x,y)ad_{s,l}(\beta(y))(-[\alpha^{-1}\beta(z),\alpha^{s+2}\beta^{l-1}(x)])\\
  &=& \varepsilon(x,y+z)\varepsilon(y,z)[\alpha^{-1}\beta[\alpha^{-1}\beta(z),\alpha^{s+1}\beta^{l-1}(y)],\alpha^{s+1}\beta^{l-1}\alpha\beta(x)]\\
  && - \varepsilon(x+y,z)[\alpha^{-1}\beta[\alpha^{-1}\beta(z),\alpha^{s+2}\beta^{l-1}(x)],\alpha^{s+1}\beta^{l-1}\beta(y)]\\
   &=&-\varepsilon(x+y,z)\varepsilon(x,y)[\beta[\alpha^{-2}\beta(z),\alpha^{s}\beta^{l-1}(y)],\alpha^{s+2}\beta^{l}(x)]
   -\varepsilon(x+y,z)[\beta[\alpha^{-2}\beta(z),\alpha^{s+1}\beta^{l-1}(x)],\alpha^{s+1}\beta^{l}(y)]\\
   &=&\varepsilon(y,z) [\beta\alpha^{s+1}\beta^l(x),\alpha[\alpha^{-2}\beta(z),\alpha^{s}\beta^{l-1}(y)]]
   +\varepsilon(x,y+z)[\beta\alpha^{s}\beta^l(y),\alpha[\alpha^{-2}\beta(z),\alpha^{s+1}\beta^{l-1}(x)]]\\
   &=& -[\alpha^{s+1}\beta^{l+1}(x),[\alpha^{-1}\beta(z),\alpha^{s+1}\beta^{l-1}(y)]]
   +\varepsilon(x,y+z)[\alpha^{s}\beta^{l+1}(y),[\alpha^{-1}\beta(z),\alpha^{s+2}\beta^{l-1}(x)]]\\
   &=&[\alpha^{s+1}\beta^{l+1}(x),[\alpha^{s}\beta^l(y),z]]
   +\varepsilon(x,y+z)[\alpha^{s}\beta^{l+1}(y),[\alpha^{-1}\beta(z),\alpha^{s+2}\beta^{l-1}(x)]]\\
   &=& [\beta^2\alpha^{s+1}\beta^{l-1}(x),[\beta\alpha^{s}\beta^{l-1}(y),\alpha\alpha^{-1}(z)]]
   +\varepsilon(x,y+z)[\beta^2\alpha^{s}\beta^{l-1}(y),[\beta\alpha^{-1}(z),\alpha\alpha^{s+1}\beta^{l-1}(x)]]\\
   &=& -\varepsilon(x+y,z)[\beta^{2}\alpha^{-1}(z),[\beta\alpha^{s+1}\beta^{l-1}(x),\alpha\alpha^{s}\beta^{l-1}(y)]]\\
   &=& -\varepsilon(x+y,z)[\alpha^{-1}\beta^{2}(z),[\alpha^{s+1}\beta^{l}(x),\alpha^{s+1}\beta^{l-1}(y)]].
\end{eqnarray*}
\end{small}
Thus, the definition of $\alpha^{s}\beta^{l}$-adjoint representation is well defined. The proof is completed.
\end{proof}

The set of $n$-cochains on $\mathcal{A}$ with coefficients in $\mathcal{A}$, which we denote by $C^{n}_{\alpha,\beta}(\mathcal{A},\mathcal{A})$, is given by
$$C^{n}_{\alpha,\beta}(\mathcal{A},\mathcal{A})=\{f \in C^{n}(\mathcal{A},\mathcal{A}) ~:~f\circ \alpha^{\otimes n}=\alpha\circ f,~~f\circ \beta^{\otimes n}=\beta\circ f\}.$$
In particular, the set of $0$-cochains is given by
$$C^{0}_{\alpha,\beta}(\mathcal{A},\mathcal{A})=\{x \in \mathcal{H}(\mathcal{A})~:~\alpha(x)=x,~\beta(x)=x\}.$$

Now, we aim to study cochain complexes associated to  $\alpha^{s}\beta^l$-adjoint representations of a BiHom-Lie colour algebra $(\mathcal{A},[.,.],\varepsilon,\alpha,\beta)$.

\begin{prop}\label{deriv} Associated to the $\alpha^{s}\beta^l$-adjoint representation $ad_{s,l}$ of the BiHom-Lie colour algebra $(\mathcal{A},[.,.],\varepsilon,\alpha,\beta)$,
$D \in C^{1}_{\alpha,\beta}(\mathcal{A},\mathcal{A})$ is $1$-cocycle of degree $\gamma$ if and only if $D$ is an $\alpha^{s+2}\beta^{l-1}$-derivation of degree $\gamma$ (i.e. $D\in (Der_{\alpha^{s+2}\beta^{l-1}}(\mathcal{A}))_{\gamma})$.
\end{prop}
\begin{proof} The conclusion follows directly from the definition of the coboundary $\delta$. $D$ is closed if and only if
$$\delta (D)(x,y)=-D([\alpha^{-1}\beta(x),y])+\varepsilon(\gamma,x)[\alpha^{s+1}\beta^l(x),D(y)]-\varepsilon(\gamma+x,y)[\alpha^{s+1}\beta^l(y),D(x)]=0.$$
So $$D([\alpha^{-1}\beta(x),y])=[D(x),\alpha^{s+1}\beta^l(y)]+\varepsilon(\gamma,x)[\alpha^{s+1}\beta^l(x),D(y)]$$
which implies that $D$ is an $\alpha^{s+2}\beta^{l-1}$-derivation of $(\mathcal{A},[.,.],\varepsilon,\alpha,\beta)$ of degree $\gamma$.
\end{proof}

\begin{prop} Associated to the $\alpha^{s}\beta^l$-adjoint representation $ad_{s,l}$, we have
\begin{eqnarray*}
H^{0}(\mathcal{A},\mathcal{A})&=&\{x \in \mathcal{H}(\mathcal{A}) ~:~ \alpha(x)=x,~~\beta(x)=x,~~[x,y]=0\}.\\
H^{1}(\mathcal{A},\mathcal{A})&=& \frac{Der_{\alpha^{s+2}\beta^{l-1}}(\mathcal{A})}{Inn_{\alpha^{s+2}\beta^{l-1}}(\mathcal{A})}.
\end{eqnarray*}
\end{prop}
\begin{proof} For any $0$-BiHom cochain $x\in C^{0}_{\alpha,\beta}(\mathcal{A},\mathcal{A})$, we have $d_{s,l}(x)(y)=-[\alpha^{s+1}\beta^{l-1}(y),x]=[\alpha^{-1}\beta(x),\alpha^{s+2}\beta^{l-2}(y)]$.\\
Therefore, $x$ is a closed $0$-BiHom-cochain if and only if
$$[\alpha^{-1}\beta(x),\alpha^{s+2}\beta^{l-2}(y)]=0,$$
which is equivalent to
$$\alpha^{-s-2}\beta^{-l+2}\Big( [\alpha^{-1}\beta(x),\alpha^{s+2}\beta^{l-2}(y)]\Big)=[x,y]=0.$$
Therefore, the set of $0$-BiHom cocycle $Z^{0}(\mathcal{A},\mathcal{A})$ is given by
$$Z^{0}(\mathcal{A},\mathcal{A})=\{x \in C^{0}_{\alpha,\beta}(\mathcal{A},\mathcal{A}) ~:~[x,y]=0,~~\forall~~ y \in \mathcal{H}(\mathcal{A})\}.$$
As, $B^{0}(\mathcal{A},\mathcal{A})=\{0\}$, we deduce that $$ H^{0}(\mathcal{A},\mathcal{A})=\{x \in \mathcal{H}(\mathcal{A}) ~:~ \alpha(x)=x,~~\beta(x)=x,~~[x,y]=0,~~\forall~~y\in \mathcal{H}(\mathcal{A})\}.$$
By Proposition \ref{deriv}, we have $Z^1(\mathcal{A},\mathcal{A}) = Der_{\alpha^{s+2}\beta^{l-1}}(\mathcal{A})$. Furthermore, it is obvious that
any exact $1$-BiHom-cochain is of the form $-[\alpha^{s+1}\beta^{l-1}(\cdot),x]$ for some $x\in C^{0}_{\alpha,\beta}(\mathcal{A},\mathcal{A})$.
Therefore, we have $B^1(\mathcal{A},\mathcal{A}) = Inn_{\alpha^{s+1}\beta^{l-1}}(\mathcal{A})$, which implies that
\begin{eqnarray*}
H^{1}(\mathcal{A},\mathcal{A})&=& \frac{Der_{\alpha^{s+2}\beta^{l-1}}(\mathcal{A})}{Inn_{\alpha^{s+2}\beta^{l-1}}(\mathcal{A})}.
\end{eqnarray*}
\end{proof}
\subsection{The coadjoint representation $\widetilde{ad}$}
In this subsection, we explore the dual representations and coadjoint representations of BiHom-Lie colour algebras.
Let $(\mathcal{A},[.,.],\varepsilon,\alpha,\beta)$ be a BiHom-Lie colour algebra and $(V,\rho,\alpha_V,\beta_V)$ be a representation of $\mathcal{A}$. Let $V^{\ast}$ be the dual vector space of $V$. We define a linear map \\
$\widetilde{\rho}:\mathcal{A} \longrightarrow End(V^{\ast})$ by $\widetilde{\rho}(x)=-^{t}\rho(x)$.
Let $f \in  V^{\ast}, x , y \in \mathcal{H}(\mathcal{A})$ and $v \in V$. We compute the right hand side of the identity $(\ref{3.39})$.
\begin{align*}
&\widetilde{\rho}(\alpha\beta(x))\circ\widetilde{\rho}(y)-\varepsilon(x,y)\widetilde{\rho}(\beta(y))\circ\widetilde{\rho}(\alpha(x)))(f)(v) \\ &=\widetilde{\rho}(\alpha \beta(x))\Big(-\varepsilon(y,f)f\circ \rho(y)(v)\Big)-\varepsilon(x,y)\widetilde{\rho}(\beta(y))\Big(-\varepsilon(x,f) f\circ \rho(\alpha(x))(v)\Big)\\
&=-\varepsilon(x+y,f) f\Big(\rho(\alpha(x))\circ \rho(\beta(y))(v)-\varepsilon(x,y) \rho(y)\circ \rho(\alpha\beta(x))(v)\Big).
\end{align*}
On the other hand, we set that the twisted map for $\widetilde{\rho}$ is $\widetilde{\beta}=^{t}\beta$, the left hand side of $(\ref{3.39})$ writes
\begin{eqnarray*}
  \widetilde{\rho}([\beta(x),y])\circ \widetilde{\beta_M}(f)(v) &=& \widetilde{\rho}([\beta(x),y])(f\circ \beta)(v) \\
   &=& -\varepsilon(x+y,f)f\circ \beta (\rho([\beta(x),y])(v)).
\end{eqnarray*}
Therefore, we have the following Proposition:
\begin{prop} Let $(\mathcal{A},[.,.],\varepsilon,\alpha)$ be a BiHom-Lie colour algebra and $(M,\rho,\beta)$ be a representation of $\mathcal{A}$. Let $M^{\ast}$ be the dual vector space of $V$. The triple $(V^{\ast},\widetilde{\rho},\widetilde{\beta})$, where $\widetilde{\rho}:\mathcal{A} \longrightarrow End(V^{\ast})$ is given by $\widetilde{\rho}(x)=-^{t}\rho(x)$, defines a representation of BiHom-Lie colour algebra $(\mathcal{A},[.,.],\varepsilon,\alpha)$ if and only if
\begin{equation}\label{rep coad}
    \beta \circ \rho([\beta(x),y])= \rho(\alpha(x))\circ \rho(\beta(y))- \varepsilon(x,y)\rho(y)\circ \rho(\alpha\beta(x)).
\end{equation}
\end{prop}
We obtain the following characterization in the case of adjoint representation.
\begin{cor} Let $(\mathcal{A},[.,.],\varepsilon,\alpha)$ be a BiHom-Lie colour algebra and $(\mathcal{A},ad,\alpha)$ be the adjoint representation of $\mathcal{A}$, where $ad: \mathcal{A} \longrightarrow End(\mathcal{A})$. We set $\widetilde{ad}:\mathcal{A} \longrightarrow End(\mathcal{A}^{\ast})$ and $\widetilde{ad}(x)(f)=-f \circ ad(x)$. Then $(\mathcal{A}^{\ast},\widetilde{ad},\widetilde{\alpha})$ is a representation of $\mathcal{A}$ if and only if
 $$\alpha \circ ad([x,y])= ad(x)\circ ad(\alpha(y))- \varepsilon(x,y)ad(y)\circ ad(\alpha(x)),~~\forall~~ x,y \in \mathcal{H}(\mathcal{A}).$$
\end{cor}

\section{Generalized $\alpha^{k}\beta^{l}$-Derivations of BiHom-Lie colour algebras}
The purpose of this section is to study the homogeneous generalized $\alpha^{k}\beta^{l}$-derivations and homogeneous $\alpha^{k}\beta^{l}$-centroid of  BiHom-Lie colour algebras, as well as $\alpha^{k}\beta^{l}$-quasi-derivations and $\alpha^{k}\beta^{l}$-quasi-centroid. The homogeneous generalized derivations were discussed first  in \cite{LIn ni}. \\

 Let $(\mathcal{A},[.,.],\varepsilon,\alpha,\beta)$ be a BiHom-Lie colour algebra.
We set $Pl_{\gamma}(\mathcal{A})=\{D \in End(\mathcal{A}):~~D(\mathcal{A}_{\gamma})\subset \mathcal{A}_{\gamma+\tau}$ for all $\tau \in \Gamma \}$.

 It turns out that  $\Big(Pl(\mathcal{A})=\bigoplus_{\gamma \in \Gamma}Pl_{\gamma}(\mathcal{A}), [.,.],\alpha\Big)$ is a BiHom-Lie colour algebra with the  Lie colour bracket
$$[D_{\gamma},D_{\tau}]=D_{\gamma}\circ D_{\tau}-\varepsilon(\gamma,\tau)D_{\tau}\circ D_{\gamma}$$
for all $D_{\gamma},D_{\tau} \in \mathcal{H}(Pl(\mathcal{A}))$ and with $\alpha:\mathcal{A} \longrightarrow \mathcal{A}$ is an even homomorphism.\\
A homogeneous $\alpha^{k}\beta^{l}$-derivation of degree $\gamma$ of $\mathcal{A}$ is an endomorphism $D \in Pl_{\gamma}(\mathcal{A})$ such that
\begin{eqnarray*}
 & [D,\alpha]=0,~~[D,\beta]=0,\\
 & D([x,y])=[D(x),\alpha^{k}\beta^{l}(y)]+\varepsilon(\gamma,x)[\alpha^{k}\beta^{l}(x),D(y)],~~\forall x,y \in \mathcal{H}(\mathcal{A}).
\end{eqnarray*}
We denote the set of all homogeneous $\alpha^{k}\beta^{l}$-derivations of degree $\gamma$ of $\mathcal{A}$ by $Der^{\gamma}_{\alpha^{k}\beta^{l}}(\mathcal{A})$. We set
$$
Der_{\alpha^{k}\beta^{l}}(\mathcal{A})=
\bigoplus_{\gamma \in \Gamma}
Der _{ \alpha ^{k}\beta^{l}} ^ {\gamma }  (\mathcal{A} ),  \
Der(\mathcal{A})=\bigoplus_{k,l\geq 0}Der_{\alpha^{k}\beta^{l}}(\mathcal{A}).$$
The set $Der(\mathcal{A})$ provided with the colour-commutator is a Lie colour algebra. Indeed, the fact that $Der_{\alpha^{k}\beta^{l}}(\mathcal{A})$ is $\Gamma$-graded implies that $Der(\mathcal{A})$ is $\Gamma$-graded
$$(Der(\mathcal{A}))_{\gamma }=\bigoplus_{k,l \geq 0}(Der_{\alpha^{k}\beta^{l}}(\mathcal{A}))_{\gamma},~~\forall~~ \gamma \in \Gamma.$$

\begin{df}
\item
\begin{enumerate}
\item A linear mapping $D\in End(\mathcal{A})$ is said to be an $\alpha^{k}\beta^{l}$-generalized derivation of degree
$\gamma$ of $\mathcal{A}$ if there exist linear mappings $D', D'' \in End(\mathcal{A})$ of degree $\gamma$ such that
\begin{eqnarray*}
& [D,\alpha] = 0,~~[D^{'},\alpha]=0,~~[D^{''},\alpha]=0,[D,\beta]=0,~~[D^{'},\beta]=0,~~[D^{''},\beta]=0,\\
& D^{''}([x,y])= [D(x),\alpha^{k}\beta^{l}(y)]+\varepsilon(\gamma,x)[\alpha^{k}\beta^{l}(x),D^{'}(y)],~~\forall x,y \in \mathcal{H}(\mathcal{A})
\end{eqnarray*}
\item A linear mapping $D\in End(\mathcal{A})$ is said to be an $\alpha^{k}\beta^{l}$-quasi-derivation of degree
$\gamma$ of $\mathcal{A}$ if there exist linear mappings $D' \in End(\mathcal{A})$ of degree $\gamma$ such that
\begin{eqnarray*}
& [D,\alpha] = 0,~~[D^{'},\alpha]=0, [D,\beta]=0,~~[D^{'},\beta]=0,\\
& D^{'}([x,y])= [D(x),\alpha^{k}\beta^{l}(y)]+\varepsilon(\gamma,x)[\alpha^{k}\beta^{l}(x),D(y)],~~\forall~~ x,y \in \mathcal{H}(\mathcal{A}).
\end{eqnarray*}
\end{enumerate}
\end{df}
The sets of generalized derivations and quasi-derivations will be denoted by $GDer(\mathcal{A})$
and $QDer(\mathcal{A})$, respectively.
\begin{df}\item
\begin{enumerate}
\item The set $C(\mathcal{A})$ consisting of linear mappings $D$ with the property
\begin{eqnarray*}
&& [D,\alpha] = 0,\\
&& D([x,y])=[D(x),\alpha^{k}\beta^{l}(y)]=\varepsilon(\gamma,x)[\alpha^{k}\beta^{l}(x),D(y)],~~\forall~~ x,y \in \mathcal{H}(\mathcal{A}).
\end{eqnarray*}
is called the $\alpha^{k}\beta^l$-centroid of degree $\gamma$ of $\mathcal{A}$.
\item The set $QC(\mathcal{A})$ consisting of linear mappings $D$ with the property
\begin{eqnarray*}
&& [D,\alpha] = 0,\\
&& [D(x),\alpha^{k}\beta^{l}(y)]= \varepsilon(\gamma,x)[\alpha^{k}\beta^{l}(x),D(y)],~~\forall~~ x,y \in \mathcal{H}(\mathcal{A}).
\end{eqnarray*}
is called the $\alpha^{k}\beta^l$-quasi-centroid of degree $\gamma$ of $\mathcal{A}$.
\end{enumerate}
\end{df}

\begin{prop} Let $(\mathcal{A},[.,.],\varepsilon,\alpha,\beta)$ be a multiplicative BiHom-Lie colour algebra. Then
\begin{eqnarray*}
&& [Der(\mathcal{A}),C(\mathcal{A})]\subseteq C(\mathcal{A}).
\end{eqnarray*}
\end{prop}
\begin{proof} Assume that $D_{\gamma} \in Der_{\alpha^{k}\beta^{l}}^\gamma(\mathcal{A}),~~D_{\eta} \in C_{\alpha^{s}\beta^{t}}^\eta(\mathcal{A})$. For arbitrary $x,y \in \mathcal{H}(\mathcal{A})$, we have
\begin{eqnarray*}
 [D_{\gamma}D_{\eta}(x),\alpha^{k+s}\beta^{l+t}(y)]&=& D_{\gamma}([D_{\eta}(x),\alpha^{s}\beta^{t}(y)])
-\varepsilon(\gamma,\eta+x)[\alpha^{k}\beta^{l}(D_{\eta}(x)),D_{\gamma}(\alpha^{s}\beta^{t}(y))]\\
 &=& D_{\gamma}([D_{\eta}(x),\alpha^{s}\beta^{t}(y)])
-\varepsilon(\gamma,\eta+x)[D_{\eta}(\alpha^{k}\beta^{l}(x)),D_{\gamma}(\alpha^{s}\beta^{t}(y))]\\
&=& D_{\gamma}D_{\eta}([x,y])-\varepsilon(\gamma,\eta+x)\varepsilon(\eta,x)[\alpha^{k+s}\beta^{l+t}(x),D_{\eta}D_{\gamma}(y)].
\end{eqnarray*}
and
\begin{eqnarray*}
 [D_{\eta}D_{\gamma}(x),\alpha^{k+s}\beta^{l+t}(y)]&=& D_{\eta}([D_{\gamma}(x),\alpha^{k}\beta^{l}(y)])\\
 &=& D_{\eta}D_{\gamma}([x,y])-\varepsilon(\gamma,x)[D_{\eta}(\alpha^{k}\beta^{l}(x)),D_{\gamma}(y)]\\
&=&  D_{\eta}D_{\gamma}([x,y])-\varepsilon(\gamma,x)D_{\eta}([\alpha^{k}\beta^{l}(x),D_{\gamma}(y)]).
\end{eqnarray*}

Now, let $\Delta_{\gamma^{'}} \in C_{\alpha^s \beta^{t}}^{s+t}(\mathcal{A})$ then we have$:$
\begin{align*}
\Delta_{\gamma^{'}}D^{''}_{\gamma}([x,y])
&=\Delta_{\gamma^{'}}([D_{\gamma}(x),\alpha^{k}\beta^{l}(y)]+\varepsilon(\gamma,x)[\alpha^{k}\beta^{l}(x),D^{'}_{\gamma}(y)])\\
&=[\Delta_{\gamma^{'}}D_{\gamma}(x),\alpha^{k+s}\beta^{l+t}(y)]+\varepsilon(\gamma+\gamma^{'},x)[\alpha^{k+s}\beta^{l+t}(x),\Delta_{\gamma^{'}}D^{'}_{\gamma}(y)].
\end{align*}
Then $\Delta_{\gamma^{'}} D_{\gamma} \in GDer_{\alpha^{k+s}\beta^{l+t}}(\mathcal{A})$ and is of degree $(\gamma+\gamma^{'})$.
\end{proof}
\begin{prop} $C(\mathcal{A})\subseteq QDer(\mathcal{A})$.
\end{prop}
\begin{proof}Let $D_{\gamma} \in C_{\alpha^{k}\beta^{l}}(\mathcal{A})$ and $x,y \in \mathcal{H}(\mathcal{A})$, then we have
\begin{align*}
[D_{\gamma}(x),\alpha^{k}\beta^{l}(y)]+ \varepsilon(\gamma,x)[\alpha^{k}\beta^{l}(x),D_{\gamma}(y)]
&=[D_{\gamma}(x),\alpha^{k}\beta^{l}(y)]+[D_{\gamma}(x),\alpha^{k}\beta^{l}(y)]\\
&=2 D_{\gamma}([x,y])\\
&=D^{"}_{\gamma}([x,y]).
\end{align*}
Then $D_{\gamma}\in QDer_{\alpha^{k}\beta^{l}}^{\gamma}(\mathcal{A})$.
\end{proof}

\begin{prop}$[QC(\mathcal{A}),QC(\mathcal{A})]\subseteq QDer(\mathcal{A})$.
\end{prop}
\begin{proof} Assume that $D_{\gamma} \in \mathcal{H}(QC_{\alpha^{k}\beta^{l}}(\mathcal{A}))$ and $D_{\tau}\in \mathcal{H}(QC_{\alpha^{s}\beta^{t}}(\mathcal{A})) $. Then for all $x,y \in \mathcal{H}(\mathcal{A})$, we have
 $$[D_{\gamma}(x),\alpha^{k}\beta^{l}(y)]=\varepsilon(\gamma,x)[\alpha^{k}\beta^{l}(x),D_{\gamma}(y)]$$
and  $$[D_{\tau}(x),\alpha^{s}\beta^{t}(y)]=\varepsilon(\tau,x)[\alpha^{s}\beta^{t}(x),D_{\tau}(y)].$$
Hence,
on the other hand, we have
\begin{eqnarray*}
[[D_{\gamma},D_{\tau}](x),\alpha^{k+s}\beta^{l+t}(y)]&=&[(D_{\gamma}\circ D_{\tau}-\varepsilon(\gamma,\tau)D_{\tau}\circ D_{\gamma})(x),\alpha^{k+s}\beta^{l+t}(y)]\\
&=&[D_{\gamma}\circ D_{\tau}(x),\alpha^{k+s}\beta^{l+t}(y)]-\varepsilon(\gamma,\tau)[D_{\tau}\circ D_{\gamma}(x),\alpha^{k+s}\beta^{l+t}(y)]\\
&=&\varepsilon(\gamma+\tau,x)[\alpha^{k+s}\beta^{l+t}(x),D_{\gamma}\circ D_{\tau}(y)]\\
&&-\varepsilon(\gamma,\tau)\varepsilon(\gamma+\tau,x)[\alpha^{k+s}\beta^{l+t}(x),D_{\tau}\circ D_{\gamma}(y)]\\
&=&\varepsilon(\gamma+\tau,x)[\alpha^{k+s}\beta^{l+t}(x),[D_{\gamma},D_{\tau}](y)]+[[D_{\gamma},D_{\tau}](x),\alpha^{k+s}\beta^{l+t}(y)],
\end{eqnarray*}
which implies that $[[D_{\gamma},D_{\tau}](x),\alpha^{k+s}\beta^{l+t}(y)]+[[D_{\gamma},D_{\tau}](x),\alpha^{k+s}\beta^{l+t}(y)]=0.$\\
Then $[D_{\gamma},D_{\tau}] \in GDer_{\alpha^{k+s}\beta^{l+t}}(\mathcal{A})$ and is of degree $(\gamma+\tau).$
\end{proof}

\subsection{BiHom-Jordan colour algebras and Derivations}
We show in the following that $\alpha^{-1}\beta^2$-derivations of BiHom-Lie colour algebras give rise to BiHom-Jordan colour algebras. First we introduce a definition of colour BiHom-Jordan algebra.
\begin{df}

A colour BiHom-algebra $(\mathcal{A},\mu,\varepsilon,\alpha,\beta)$ is a BiHom-Jordan colour algebra if hold  the identities
    \begin{eqnarray}
& \alpha\circ \beta=\beta\circ \alpha,\\
& \mu(\beta(x),\alpha(y))= \varepsilon(x,y)\mu(\beta(y),\alpha(x)),\\
& \circlearrowleft_{x,y,w}\varepsilon(w,x+z) as_{\alpha,\beta}\Big(\mu(\beta^2(x),\alpha\beta(y)),\alpha^2\beta(z),\alpha^3(w)\Big)=0.\label{BiHomJordanId}\end{eqnarray}
 for all $x,y,z$ and $w$ in $\mathcal{H}(\mathcal{A}) $ and where $as_{\alpha,\beta}$ is the BiHom-associator defined in \eqref{BiHomAssociator}.\\
 The identity \eqref{BiHomJordanId} is called BiHom-Jordan colour  identity.

\end{df}
Observe that when $\beta=\alpha$, the BiHom-Jordan colour identity reduces to the Hom-Jordan colour identity.
\begin{prop}\label{DERhomJORDAN} Let $(\mathcal{A},[.,.],\varepsilon,\alpha,\beta)$ be a multiplicative BiHom-Lie colour algebra. Consider the operation
$D_{1}\bullet D_{2}=D_{1}\circ D_{2}+\varepsilon(d_1,d_2)D_{2}\circ D_{1}$ for all $\alpha^{-1} \beta^2$-derivations $D_{1}, D_{2} \in \mathcal{H}(Pl(\mathcal{A}))$. Then the $5$-tuple $(Pl(\mathcal{A}),\bullet,\varepsilon,\alpha,\beta)$ is a BiHom-Jordan colour algebra.
\end{prop}
\begin{proof} Assume that $D_{1},D_{2}, D_{3},D_{4} \in \mathcal{H}(Pl(\mathcal{A}))$, we have
\begin{eqnarray*}
 \beta(D_{1})\bullet \alpha(D_{2})&=& \beta(D_{1})\circ \alpha(D_{2})+\varepsilon(d_1,d_2)\beta(D_{2})\circ \alpha(D_{1})  \\
   &=&\varepsilon(d_1,d_2)(\beta(D_{2})\circ \alpha(D_{1})+\varepsilon(d_2,d_1))\beta(D_{1})\circ \alpha(D_{2})  \\
   &=&\varepsilon(d_1,d_2)\beta(D_{2})\bullet \alpha(D_{1}).
\end{eqnarray*}
 Since
\begin{align*}
&\Big((\beta^2(D_{1})\bullet \alpha\beta(D{2}))\bullet \alpha^2\beta( D_{3})\Big)\bullet \beta\alpha^{3}(D_{4})\\
&=\Big((\beta^2(D_1) \alpha\beta(D_2))\alpha^2\beta(D_3)\Big) \beta\alpha^{3}(D_4)+\varepsilon(d_1+d_2+d_3,d_4)\beta\alpha^{3}(D_4)\Big(\beta^2(D_1)\alpha\beta (D_2))\alpha^2\beta(D_3)\Big)\\
&+\varepsilon(d_1+d_2,d_3)\Big(\alpha^2\beta(D_3)(\beta^2(D_1)\alpha\beta(D_2))\Big) \beta\alpha^{3}(D_4)\\
&+\varepsilon(d_1+d_2,d_3)\varepsilon(d_1+d_2+d_3,d_4)\beta\alpha^{3}(D_4)\Big(\alpha^2\beta(D_3)(\beta^2(D_1) \alpha\beta(D_2))\Big)\\&+\varepsilon(d_1,d_2)(\alpha\beta(D_2)\beta^2(D_1))\alpha^2\beta(D_3))\beta\alpha^{3}(D_4)\\
&+\varepsilon(d_1,d_2)\varepsilon(d_1+d_2+d_3,d_4)\beta\alpha^{3}(D_4)\Big( (\alpha\beta(D_2)\beta^2(D_1))\alpha^2\beta(D_3)\Big)\\
&+\varepsilon(d_1,d_2)\varepsilon(d_1+d_2,d_3)\Big( \alpha^2\beta(D_3)(\alpha\beta(D_3)\beta^2(D_1))\Big) \beta\alpha^{3}(D_4)\\
&+\varepsilon(d_1,d_2)\varepsilon(d_1+d_2,d_3)\varepsilon(d_1+d_2+d_3,d_4)\beta\alpha^{3}(D_4)\Big(\alpha^2\beta(D_3)(\alpha\beta(D_2)\beta^2(D_1))\Big)
\end{align*}
and
\begin{align*}
&\alpha(\beta^2(D_{1})\bullet \alpha\beta(D_{2})) \bullet (\beta\alpha^2(D_{3})\bullet\alpha^3(D_{4}))\\
&=(\alpha\beta^2(D_1) \alpha^2\beta(D_2)) (\alpha^2\beta(D_3)\alpha^3(D_4))+\varepsilon(d_1+d_2,d_3+d_4)(\alpha^2\beta(D_3)\alpha^3(D_4))(\alpha\beta^2(D_1)\alpha^2\beta (D_2))\\
&+\varepsilon(d_3,d_4)(\alpha\beta^2(D_1) \alpha^2\beta(D_2))(\alpha^3(D_4)\alpha^2\beta(D_3))\\
&+\varepsilon(d_3,d_4)\varepsilon(d_1+d_2,d_3+d_4)(\alpha^3(D_4)\alpha^2\beta(D_{3}))(\alpha\beta^2(D_1)\alpha^2\beta(D_2))\\
&+\varepsilon(d_1,d_2)(\alpha^2\beta(D_2) \alpha\beta^2(D_1))(\alpha^2\beta(D_3)\alpha\beta^2(D_4))\\
&+\varepsilon(d_1,d_2)\varepsilon(d_1+d_2,d_3+d_4)(\alpha^2\beta(D_3)\alpha^3(D_4))(\alpha^2\beta(D_2)\alpha\beta^2 D_1))\\
&+\varepsilon(d_1,d_2)\varepsilon(d_3,d_4)(\alpha^2\beta(D_2)\alpha\beta^2(D_1))(\alpha^3(D_4)\alpha^2\beta(D_3))\\
&+\varepsilon(d_1,d_2)\varepsilon(d_3,d_4)\varepsilon(d_1+d_2,d_3+d_4)(\alpha^3(D_4)\alpha^2\beta(D_3))(\alpha^2\beta(D_2)\alpha\beta^2(D_1)).
\end{align*}
Then we have
\begin{align*}
&\varepsilon(d_4,d_1+d_3) as_{\alpha,\beta}\Big(\beta^2(D_{1})\bullet \alpha\beta(D_2),\alpha\beta(D_3),\alpha^3(D_4)\Big)\\
&=\varepsilon(d_2,d_3)\varepsilon(d_1+d_2+d_4,d_3)(\alpha^2\beta(D_3)\alpha^3(D_4))(\alpha\beta^2(D_1)\alpha^2\beta (D_2))\\
&+\varepsilon(d_4,d_1)(\alpha\beta^2(D_1)\alpha^2\beta(D_2))(\alpha^3(D_4)\alpha^2\beta(D_3))\\
&+\varepsilon(d_2,d_4)\varepsilon(d_1,d_2)\varepsilon(d_1+d_2+d_4,d_3)(\alpha^2\beta(D_3)\alpha^3(D_4))(\alpha^2\beta(D_2)\alpha\beta^2 D_1))\\
&+\varepsilon(d_4,d_1)\varepsilon(d_1,d_2)(\alpha^2\beta(D_2)\alpha\beta^2(D_1))(\alpha^3(D_4)\alpha^2\beta(D_3))\\
&-\varepsilon(d_2,d_4)\beta\alpha^{3}(D_4)\Big(\beta^2(D_1)\alpha\beta (D_2))\alpha^2\beta(D_3)\Big)\\
&-\varepsilon(d_4,d_1+d_3)\varepsilon(d_1+d_2,d_3)\Big(\alpha^2\beta(D_3)(\beta^2(D_1)\alpha\beta(D_2))\Big) \beta\alpha^{3}(D_4)\\
&-\varepsilon(d_1,d_2)\varepsilon(d_2,d_4)\beta\alpha^{3}(D_4)\Big( (\alpha\beta(D_2)\beta^2(D_1))\alpha^2\beta(D_3)\Big)\\
&-\varepsilon(d_4,d_1+d_3)\varepsilon(d_1,d_2)\varepsilon(d_1+d_2,d_3)\Big( \alpha^2\beta(D_3)(\alpha\beta(D_3)\beta^2(D_1))\Big) \beta\alpha^{3}(D_4).
\end{align*}
Therefore, we get
\begin{align*}
&\circlearrowleft_{D_1,D_3,D_4}\varepsilon(d_4,d_1+d_3)as_{\alpha,\beta}\Big(\beta^2(D_{1})\bullet \alpha\beta(D_2),\alpha\beta(D_3),\alpha^3(D_4)\Big)=0,
\end{align*}
and so the statement holds.
\end{proof}
\begin{cor} Let $(\mathcal{A},[.,.],\varepsilon,\alpha,\beta)$ be a multiplicative BiHom-Lie colour algebra. Consider  the operation
$$D_{1}\bullet D_{2}=D_{1}\circ D_{2}-\varepsilon(d_1,d_2)D_{2}\circ D_{1}$$ for all $D_{1}, D_{2} \in \mathcal{H}(QC(\mathcal{A}))$. Then the $5$-tuple $(QC(\mathcal{A}),\bullet,\varepsilon,\alpha,\beta)$ is a BiHom-Jordan colour algebra.
\end{cor}
\begin{proof} We need only to show that $D_{1}\bullet D_{2} \in QC(\mathcal{A})$, for all $D_{1}, D_{2} \in \mathcal{H}(QC(\mathcal{A})).$\\
Assume that $x,y \in \mathcal{H}(\mathcal{A})$, we have
\begin{align*}
&[D_{1}\bullet D_{2}(x),\alpha^{k+s}\beta^{l+t}(y)]\\
&=[D_{1}\circ D_{2}(x),\alpha^{k+s}\beta^{l+t}(y)]+\varepsilon(d_1,d_2)[D_{2}\circ D_{1}(x),\alpha^{k+s}\beta^{l+t}(y)]\\
&=\varepsilon(d_1,d_2+x)[D_{2}(x),D_{1}(\alpha^{k+s}\beta^{l+t}(y))]+\varepsilon(d_2,x)[D_{1}(x),D_{2}(\alpha^{k+s}\beta^{l+t}(y))] \\
&=\varepsilon(d_1,d_2)\varepsilon(d_1+d_2,x)[\alpha^{k+s}\beta^{l+t}(x),D_{2}\circ D_{1}(y)]+\varepsilon(d_1+d_2,x)[\alpha^{k+s}\beta^{l+t}(x),D_{1}\circ D_{1}(y)] \\
&=\varepsilon(d_1+d_2,x)[\alpha^{k+s}\beta^{l+t}(x),D_{1}\bullet D_{2}(y)].
\end{align*}
Hence $D_{1}\bullet D_{2} \in QC(\mathcal{A}).$
\end{proof}

\end{document}